\documentclass[reqno]{amsart}

\usepackage{hyperref}
\usepackage{amsmath,amsthm,amssymb,amscd}
\usepackage{todonotes}

\usepackage{tikz}
\usetikzlibrary{angles,calc,arrows.meta}

\usepackage[
  style=numeric,
  sorting=nyt,
  giveninits=true,
  doi=false,
  url=false,
  isbn=false,
  maxnames=10,
]{biblatex}
\DeclareFieldFormat[article,incollection,inproceedings,misc]{title}{#1}
\DeclareFieldFormat[article,incollection,inproceedings,misc]{volume}{\mkbibbold{#1}}
\DeclareFieldFormat{pages}{#1}
\renewbibmacro{in:}{}
\AtEveryBibitem{\ifentrytype{book}{\clearfield{pages}}{}}
\AtEveryBibitem{\clearfield{eprintclass}}
\AtEveryBibitem{\clearlist{language}}

\newtheorem{thm}{Theorem}
\newtheorem{prop}[thm]{Proposition}
\newtheorem{cor}[thm]{Corollary}
\newtheorem{lem}[thm]{Lemma}

\theoremstyle{definition}
\newtheorem{dfn}[thm]{Definition}
\newtheorem{rem}[thm]{Remark}
\newtheorem{ex}[thm]{Example}
\newtheorem{prob}[thm]{Problem}
\newtheorem{ques}[thm]{Question}

\newtheorem{conv}[thm]{Convention}

\theoremstyle{remark}
\newtheorem*{org}{Organization}
\newtheorem*{ack}{Acknowledgements}

\numberwithin{thm}{section}
\numberwithin{equation}{section}

\title[Finsler structure of G-spaces]{Finsler structure of Busemann G-spaces}
\author[T. Fujioka]{Tadashi Fujioka}
\address[T. Fujioka]{Department of Applied Mathematics, Fukuoka University, Fukuoka 814-0180, Japan}
\email{tfujioka210@gmail.com}

\author[S. Gu]{Shijie Gu}
\address[S. Gu]{Department of Mathematics, Northeastern University, Shenyang, Liaoning, China, 110004}
\email{shijiegutop@gmail.com}

\date{\today}

\begin{document}

\begin{abstract}
We provide two sufficient conditions for a Busemann G-space to admit a differentiable DC atlas with a continuous Finsler metric, from the viewpoint of comparison geometry.
These results generalize previous work on G-spaces with Riemannian curvature bounds, namely the Alexandrov and CAT conditions, to the Finsler setting.
\end{abstract}

\maketitle
\tableofcontents

\section{Introduction}\label{sec:intro}

\subsection{Main results}

A \emph{G-space}, introduced by Busemann \cite{Bu55}, is a qualitative generalization of Finsler manifolds with an injectivity radius bound.
It is defined as a locally compact, complete geodesic space in which any shortest path admits a local unique extension.
For the precise definition, see Section \ref{sec:pre}.
For recent surveys on this subject (focusing on its topological aspects), we refer to \cite{BHR11, An18}.

Any sufficiently regular Finsler manifold is, indeed, a G-space (\cite[Section 15]{Bu55}, \cite{BM41}, \cite{Po90sol}).
The aim of this paper is to prove a kind of the converse statement: we provide sufficient (and as minimal as possible) conditions for a G-space to admit a Finsler structure.
Such conditions were already studied by Berestovskii and Nikolaev \cite{Be75,Ni83} (see also \cite{ABN86,BN93} and references therein), Berestovskii \cite{Be94, Be02}, and especially Pogorelov \cite{Po90rie, Po90reg, Po98}.
However, their assumptions contain Riemannian curvature bounds, two-sided curvature bounds, or the continuous differentiability of the distance function.
Our approach, motivated by recent studies of Busemann convex/concave spaces \cite{FG25, HY25}, focuses on Finsler-type one-sided curvature bounds.
See Section \ref{sec:rel} for related results.

More specifically, we prove the following two theorems.
These are Finslerian generalizations of the previous results for Alexandrov/locally CAT G-spaces in \cite{Be94,OS94,Per94}/\cite{Be02,OT99,LN19} (see Section \ref{sec:rel}).

\begin{thm}\label{thm:conc}
Let $X$ be a locally semiconcave G-space.
Then $X$ is a topological manifold with a canonical $C^{1,1/2}$ and DC atlas and a continuous Finsler metric that is compatible with the original distance.
Furthermore, the Finsler metric is locally $1/2$-H\"older continuous in the base direction.
\end{thm}

\begin{thm}\label{thm:conv}
Let $X$ be a locally semiconvex, locally differentiable G-space.
Then $X$ is a topological manifold with a canonical $C^1$ and DC atlas and a continuous Finsler metric that is compatible with the original distance.
\end{thm}

Let us briefly explain the meaning of the assumptions and conclusions.
We say that a geodesic space $X$ is \emph{locally semiconcave}/\emph{semiconvex} if every point $x\in X$ has a neighborhood $U$ such that for any $p\in U$ and any unit-speed shortest path $\gamma(t)$ in $U$, the squared distance function
\[t\mapsto d(p,\gamma(t))^2\]
is uniformly semiconcave/semiconvex (independent of $p$ and $\gamma$).
We also say that $X$ is \emph{locally differentiable} if the above squared distance function is differentiable.
See Section \ref{sec:pre} for more details.

An atlas of a topological manifold is said to be \emph{$C^{1,\alpha}$}/\emph{DC}, where $0\le\alpha\le1$ and $C^{1,0}=C^1$, if the coordinate transformations preserve the class of $C^{1,\alpha}$/DC functions (the latter class consists of functions represented as differences of semiconcave/semiconvex functions).
For a geodesic space $X$, we say that a $C^{1,\alpha}$/DC atlas on $X$ is \emph{canonical} if any squared distance function as above is $C^{1,\alpha}$/DC in local coordinates.
For a $C^1$ manifold $X$, a \emph{Finsler metric} is a function
\[F:TX\to\mathbb R\]
defined on the tangent bundle of $X$ that is a norm with respect to the fiber variable.
The conclusions of the main theorems further require that $F$ is continuous on $TX$ and induces the original distance of $X$, and that it is locally $1/2$-H\"older continuous with respect to the base variable in the case of Theorem \ref{thm:conc}.
See Sections \ref{sec:dc} and \ref{sec:fin} for more details.

The local semiconcavity and local semiconvexity are weak forms of lower and upper curvature bounds, respectively, in the Finsler setting.
Indeed, for smooth Finsler manifolds, these conditions are described by the flag and tangent curvatures, as well as the uniform smoothness and convexity constants; see \cite{Oh09uni, Oh21, Sh01}.

From a technical point of view, our proofs of Theorems \ref{thm:conc} and \ref{thm:conv} are not particularly new and are similar to previous ones in various settings: we construct a coordinate chart consisting of distance functions (i.e., strainer coordinates) and define a Finsler metric by pulling back the derivative of the distance function.
However, the purpose of this paper is not to provide new techniques, but rather to clarify, from the standpoint of comparison Finsler geometry, what conditions are really necessary for the existing argument (partly for the sake of future research in \cite{Fgeo}).
The following two remarks illustrate this perspective well.

\begin{rem}\label{rem:main}
The local semiconcavity in Theorem \ref{thm:conc} (together with the G-space assumption) implies local differentiability, see Lemma \ref{lem:diff}.
On the other hand, the local semiconvexity in Theorem \ref{thm:conv} alone (even together with the G-space assumption) does not imply local differentiability; a counterexample is a strictly convex and non-differentiable norm on Euclidean space.
This is a typical difference between lower and upper curvature bounds in the Finsler setting, stemming from the one-sided nature of the triangle inequality (Remark \ref{rem:diff2}).
We remark that the importance of such differentiability was already suggested in the original works of Busemann and Phadke, see \cite[Chapter III]{Bu55} and \cite{BP79} (however, the meaning of differentiability is somewhat different).
\end{rem}

\begin{rem}\label{rem:os}
As mentioned above, Theorems \ref{thm:conc} and \ref{thm:conv} generalize the corresponding results for Alexandrov and locally CAT G-spaces, respectively (see Section \ref{sec:rel}).
In the Alexandrov case, the $1/2$-H\"older continuity of the Riemannian structure is due to Otsu--Shioya \cite[Lemma 3.2(2)]{OS94}.
The proof of Theorem \ref{thm:conc}, which does not rely on the Riemannian aspects of a lower curvature bound, provides a slightly streamlined explanation of this fact (see Proposition \ref{prop:c1} and Lemma \ref{lem:fin2}).
On the other hand, regarding Theorem \ref{thm:conv}, the $1/2$-H\"older continuity is not known even for locally CAT G-spaces (cf.\ \cite[Remark 1.2]{Be02}).
\end{rem}

In the general semiconvex case without differentiability, for now we only prove the following (see also Question \ref{ques:conv}).

\begin{thm}\label{thm:top}
Any locally semiconvex G-space is a topological manifold.
\end{thm}

This generalizes the previous result of Andreev \cite{An14} (cf.\ \cite{An17nor}) for locally Busemann convex G-spaces; see Section \ref{sec:rel}.
Recently, the authors \cite[Theorem 1.6]{FG25} obtained an alternative proof of Andreev's result.
The above theorem can be regarded as a generalization of this alternative proof.

\subsection{Related results and conjectures}\label{sec:rel}

We now recall several related results and conjectures.
This subsection mentions numerous conditions that will not be used in the rest of the paper, so readers unfamiliar with them may skip it.

\subsubsection*{Alexandrov/CAT spaces}
An \emph{Alexandrov space} and a \emph{CAT space} are, respectively, geodesic spaces with lower and upper sectional curvature bounds in the sense of triangle comparison.
More precisely, every geodesic triangle in an Alexandrov/CAT space is not thinner/thicker than the geodesic triangle with the same side-lengths in the model plane of constant curvature, which represents a lower/upper curvature bound of the space.
Standard references are \cite{BBI01, AKP24, BH99, AKP19}.

The Alexandrov condition implies local semiconcavity as in Theorem \ref{thm:conc}, and the local CAT condition implies local semiconvexity as in Theorem \ref{thm:conv} (note that the Alexandrov condition satisfies the local-to-global property, while the CAT condition does not, so we distinguish local and global only for the latter).
The local CAT condition, together with the G-space assumption, also implies local differentiability (see Lemma \ref{lem:cat}).
Therefore, Theorem \ref{thm:conc} covers Alexandrov G-spaces and Theorem \ref{thm:conv} covers locally CAT G-spaces.

However, the Alexandrov/local CAT condition excludes Finsler manifolds, unlike local semiconcavity/semiconvexity.
That is, any Finsler manifold satisfying the Alexandrov/local CAT condition is necessarily Riemannian.
From our perspective that G-spaces are a generalization of Finsler manifolds, this is not desirable.

The structures of Alexandrov G-spaces and locally CAT G-spaces were studied by Berestovskii \cite{Be94, Be02}.
It turns out that any Alexandrov (resp.\ locally CAT) G-space is a canonical $C^{1,1/2}$ (resp\ $C^1$) manifold with a locally $1/2$-H\"older continuous (resp.\ continuous) Riemannian metric that is compatible with the original distance.
Here the improved regularity in the Alexandrov case follows from the result of Otsu--Shioya \cite[Lemma 3.2(2)]{OS94} mentioned in Remark \ref{rem:os}.
Our results generalize Berestovskii's ones to the Finsler setting.

These results of Berestovskii can be viewed as special cases of the  general theory of singular spaces with a lower or upper curvature bound, i.e., Alexandrov spaces and GCBA spaces.
Here a \emph{GCBA space} is a separable, locally compact, locally geodesically complete, locally CAT space.
Note that an Alexandrov G-space is nothing but a (finite-dimensional) Alexandrov space with extendable geodesics and that a locally CAT G-space is nothing but a (complete) GCBA space without branching geodesics. 

The general theory tells us that a regular part of an Alexandrov/GCBA space, in some suitable sense, admits a canonical $C^{1,1/2}$/$C^1$ and DC atlas with a locally $1/2$-H\"older continuous/continuous Riemannian metric of locally bounded variation.
For Alexandrov spaces, see Burago--Gromov--Perelman \cite{BGP92}, Otsu--Shioya \cite{OS94}, and Perelman \cite{Per94}.
For GCBA spaces, see Lytchak--Nagano \cite{LN19} (cf.\ Otsu--Tanoue \cite{OT99} and Otsu \cite{Ot97}).
Restricting these results to the special case of G-spaces, one can recover and even improve Berestovskii's results.

\subsubsection*{Busemann convex/concave spaces}

The \emph{Busemann convexity} and \emph{Busemann concavity} are further Finslerian generalizations of the CAT($0$) condition and nonnegative Alexandrov curvature, respectively, distinct from the semiconvexity and semiconcavity.
More precisely, a geodesic space $X$ satisfies the Busemann convexity/concavity if for every pair of constant-speed shortest paths $\gamma,\eta:[0,1]\to X$ emanating from the same point, the function
\[t\mapsto d(\gamma(t),\eta(t))/t\]
is nondecreasing/nonincreasing in $t$.
For recent references, see \cite{Jo97, Pa14, FG25, Ke19, HY25}.
The Busemann convexity is also known as the \emph{Busemann nonpositive curvature}, abbreviated as \emph{BNPC}.
These conditions do not exclude Finsler manifolds, unlike the CAT and Alexandrov conditions (see \cite{KVK04, KK06, IL19} for the Busemann convexity and \cite[Section 2.1, Remark]{Ke19} for the Busemann concavity).

In fact, the local Busemann convexity implies the local (semi)convexity, thanks to the triangle inequality.
On the other hand, the local Busemann concavity does not necessarily imply the local semiconcavity; consider a strictly convex, non-differentiable normed space as in Remark \ref{rem:main}.
This is another typical difference between upper and lower curvature bounds in the Finsler setting.
Note that this example also shows that the local Busemann convexity does not imply local differentiability, even if it is a G-space.

Andreev \cite{An14, An17nor} studied the structure of locally BNPC G-spaces.
He proved that such a G-space is a topological manifold and that every tangent cone is isometric to a finite-dimensional strictly convex normed space.
Based on these results, Andreev introduced the notion of a \emph{singular Finsler space} in \cite{An17fin}.
Theorem \ref{thm:top} generalizes his first result \cite{An14}.
Theorem \ref{thm:conv} is similar to his second result \cite{An17nor}, but differs in several ways: regarding the assumptions, our theorem additionally assumes differentiability but instead weakens the Busemann convexity to semiconvexity; regarding the conclusions, Andreev's theorem describes a Finsler structure in terms of the tangent cone, whereas ours deals with a Finsler structure with respect to local coordinates and further shows its continuity.

Recently, the first author and Kenshiro Tashiro \cite{FT26} also studied the structure of BNPC spaces equipped with measures satisfying the measure contraction property (MCP).
It turns out that any locally BNPC space satisfying the local MCP for the Hausdorff measure is a topological manifold with boundary, and its interior is a G-space (without completeness); see \cite[Theorem 1.5]{FT26} and its proof.
Theorem \ref{thm:conv} shows that one can construct a Finsler metric on the interior, at least under the additional assumption of differentiability.

In view of the previous part on Alexandrov/CAT spaces, it is natural to expect that Theorems \ref{thm:conc}, \ref{thm:conv}, and \ref{thm:top}, as well as Andreev's results, are special cases of more general theories of singular spaces satisfying Finsler-type curvature bounds.
In what follows, we recall such studies in \cite{FG25,HY25}.

As a generalization of a GCBA space (of nonpositive curvature), the authors \cite{FG25} introduced the notion of a \emph{GNPC space}, that is, a separable, locally compact, locally geodesically complete, locally BNPC space.
In \cite{FG25} the authors studied the topological structure of GNPC spaces, extending the results of Lytchak--Nagano \cite{LN19, LN22} and Lytchak--Nagano--Stadler \cite{LNS24} for GCBA spaces.
It is likely that most of the geometric part of the Lytchak--Nagano theory can also be extended to GNPC spaces, possibly with appropriate differentiability assumptions as in Theorem \ref{thm:conv}.
In fact, one of the motivations for the present paper is to conduct an experimental research in this direction.
The general theory will be discussed in a forthcoming paper by the first author \cite{Fgeo}.

On the other hand, Han--Yin \cite{HY25} studied the structure of Busemann concave space satisfying (global) semiconcavity and local semiconvexity, and obtained parallel results to the Burago--Gromov--Perelman theory \cite{BGP92}.
Note that any Alexandrov space of nonnegative curvature satisfies the Busemann concavity and global semiconcavity (but may fail local semiconvexity).

The relationships discussed above are summarized in Figures \ref{fig:alex} and \ref{fig:cat}.
Note that the horizontal inclusions are only valid for the case of nonnegative/nonpositive curvature.
Our main theorems deal with the ones on the upper right in both diagrams.

\begin{figure}[ht]
\[
\begin{array}{ccc}
 & & \text{locally semiconcave G-spaces}\\
 & \rotatebox[origin=c]{45}{$\subset$} & \rotatebox[origin=c]{90}{$\subset$} \\
\text{Alexandrov G-spaces} & \underset{\text{nonnegative case}}{\subset} & \begin{array}{c}\text{Busemann concave,}\\\text{(globally) semiconcave G-spaces}\end{array} \\
\rotatebox[origin=c]{-90}{$\subset$} & & \rotatebox[origin=c]{-90}{$\subset$} \\
\text{Alexandrov spaces} & \underset{\text{nonnegative case}}{\subset} & \begin{array}{c}\text{Busemann concave,}\\\text{(globally) semiconcave spaces}\end{array}
\end{array}
\]
\caption{}\label{fig:alex}
\end{figure}

\begin{figure}[ht]
\[
\begin{array}{ccc}
 & & \begin{array}{c}\text{locally semiconvex,}\\\text{(locally differentiable) G-spaces}\end{array}\\
 & \rotatebox[origin=c]{45}{$\subset$} & \rotatebox[origin=c]{90}{$\subset$} \\
\text{locally CAT G-spaces} & \underset{\text{nonpositive case}}{\subset} & \begin{array}{c}\text{locally Busemann convex,}\\\text{(locally differentiable) G-spaces}\end{array}\\
\rotatebox[origin=c]{-90}{$\subset$} & & \rotatebox[origin=c]{-90}{$\subset$} \\
\text{GCBA spaces} & \underset{\text{nonpositive case}}{\subset} & \text{GNPC spaces}
\end{array}
\]
\caption{}\label{fig:cat}
\end{figure}

It is worth highlighting that Theorems \ref{thm:conc}, \ref{thm:conv}, and \ref{thm:top} do not require the Busemann concavity/convexity.
This indicates that several aspects of the boundedness of curvature, which are consistent in the Riemannian setting, appear differently in the Finslerian framework, and only some of them are necessary for our arguments.
In a sense, the other aspects can be considered to be already present in the properties of G-spaces in a qualitative form; for example, the non-branching/uniqueness property of shortest paths in a G-space can be regarded as a very weak form of the Busemann concavity/convexity.

\subsubsection*{Pogorelov's work}

Pogorelov \cite{Po90reg, Po98} gave (non-Riemannian) sufficient conditions for a G-space to be a Finsler manifold (see also \cite{Po90rie} for a Riemannian condition).
Our results are closely related to these two works.
In \cite{Po90reg}, he proved that any G-space satisfying both local semiconcavity and local semiconvexity is a Finsler manifold (of Lipschitz regularity).
Theorems \ref{thm:conc} and \ref{thm:conv} partially generalize this result to one-sided curvature bounds (it is obvious that local semiconcavity plus local semiconvexity implies local differentiability, see Lemmas \ref{lem:sum}, \ref{lem:sum'} and Definition \ref{dfn:diff'}). 
In \cite[Chapter 3]{Po98}, he proved a similar result under a weaker assumption, called \emph{Axiom A}, on the continuous differentiability of the distance function.
We verify that the G-spaces of Theorems \ref{thm:conc} and \ref{thm:conv} satisfy Axiom A or a slightly weaker version of it; see Remarks \ref{rem:pog1} and \ref{rem:pog2}.
However, since the conclusions of our theorems contain more than his theorem, and furthermore, his proof is a little bit sketchy (see, for example, Remark \ref{rem:pog3}), we will provide a complete argument including the construction of a Finsler metric.
This clarification is in fact another small motivation for this paper (partly for future reference in \cite{Fgeo}).

\subsubsection*{Busemann's conjecture}

Busemann \cite{Bu55} originally conjectured that any G-space is finite-dimensional (in the sense of the topological dimension) and that any finite-dimensional G-space is a topological manifold.
These conjectures have remained open in their full generality.
As above, it is not very difficult to verify both conjectures under quantitative assumptions such as curvature bounds (i.e., restrictions on the second-order derivative of the distance function) or continuous first-order differentiability (see also \cite[Section 5]{Bu70}): the essential difficulty of these problems lies in their qualitative nature.

The finite-dimensionality conjecture was proved under a qualitative assumption of the convexity of metric balls, see \cite{Be77, BHR11}.
The manifold conjecture was known to be true in dimensions $\le 4$, \cite{Bu55, Krak68, Th96}, and remains open in dimension $\ge 5$.

In general, it is only known that any G-space is \emph{topologically homogeneous}, i.e.,  any point can be moved to any other point by a self-homeomorphism of the space (\cite{Th96, BHR11}).
Under the finite-dimensionality assumption, a little more is known: any finite-dimensional G-space is an ANR homology manifold (\cite{Th96, BHR11}) and every small metric sphere has the homotopy type of a sphere (\cite{Gu19}).
In particular, the manifold conjecture follows from the solution of a more general topological problem called the \emph{Bing--Borsuk conjecture}, see \cite{HR08}.

Although the general problems are purely topological, to attack them, one must face the Finsler character of G-spaces.
The proofs of the main theorems (especially Theorem \ref{thm:top}) might shed some light on the potential difficulties arising from those geometric aspects of G-spaces.
See Problems \ref{prob:ball} and \ref{prob:ball'}.

\begin{org}
In Section \ref{sec:pre}, we define the basic terminology used in this paper.
In Section \ref{sec:ang}, we introduce the notion of angle, which plays a fundamental role in the proofs of the main theorems.
In Section \ref{sec:conc}, we prove Theorem \ref{thm:conc}.
In Section \ref{sec:conv}, we prove Theorem \ref{thm:conv} in a parallel manner.
In Section \ref{sec:top}, modifying the previous argument, we prove Theorem \ref{thm:top}.
Finally, in Section \ref{sec:prob}, we discuss open problems.
\end{org}

\begin{ack}
The authors would like to thank Liming Yin and Bang-Xian Han for fruitful discussions.
The first author is also grateful to Kenshiro Tashiro, Shouhei Honda, and Shin-ichi Ohta for helpful comments.
The first author was supported by JSPS KAKENHI Grant Number 25K23336.
The second author was supported by NSFC grant 12201102.
\end{ack}

\section{Preliminaries}\label{sec:pre}

In this section, we fix the basic terminology used in this paper, i.e., G-spaces, local semiconcavity/semiconvexity, local differentiability, and normal balls.

Throughout the paper, the distance between $x$ and $y$ is denoted by $d(x,y)$ or $|xy|$.
The open and closed balls around $p$ of radius $r$ are denoted by $B(p,r)$ and $\bar B(p,r)$, respectively.

A \emph{shortest path} in a metric space is an isometric embedding from an interval.
In particular, any shortest path has unit-speed.
We denote by $xy$ a shortest path from $x$ to $y$.
A \emph{geodesic space} is a metric space such that every two points are joined by at least one shortest path.

Unless otherwise stated, we will use the standard Euclidean norm for $\mathbb R^n$.

\subsection{G-spaces}\label{sec:g}

We begin by recalling the original axiomatic definition of a G-space by Busemann \cite{Bu55}.

\begin{dfn}
A metric space $X$ is called a \emph{G-space} if it satisfies the following four axioms:
\begin{enumerate}
    \item For any distinct points $x,y\in X$, there exists $z\in X\setminus\{x,y\}$ such that $|xz|+|zy|=|xy|$.
    \item Every bounded infinite set in $X$ has an accumulation point.
    \item Every point of $X$ has an open neighborhood $U$ such that for any $x,y\in U$, there exists $z\in U\setminus\{x,y\}$ such that $|xy|+|yz|=|xz|$.
    \item For any distinct points $x,y\in X$, if $z_1,z_2\in X$ satisfy $|xy|+|yz_i|=|xz_i|$ and $|yz_1| =|yz_2|$, then $z_1=z_2$.
\end{enumerate}
\end{dfn}

Conditions (1) and (2) imply that $X$ is a complete, locally compact geodesic space.
Moreover, Conditions (3) and (4) imply that every shortest path is locally uniquely extendable.
Therefore the above definition is equivalent to the following non-axiomatic definition.

\begin{dfn}\label{dfn:g}
A complete, locally compact geodesic space $X$ is a \emph{G-space} if every point has an open neighborhood $U$ such that any shortest path $\gamma$ in $U$ admits a unique maximal extension reaching $\partial U$.
\end{dfn}

Note that the above unique extension property implies the local uniqueness of shortest paths: any two points in a sufficiently small neighborhood in a G-space is joined by a unique shortest path.

Every G-space is \emph{topologically homogeneous}, i.e.,  any point can be moved to any other point by a self-homeomorphism of the space (\cite{Th96, BHR11}).
Moreover, every finite-dimensional G-space is an ANR homology manifold (\cite{Th96, BHR11}) such that any small metric sphere has the homotopy type of a sphere (\cite{Gu19}).
However, in general, nothing more is known about the topological structure of G-spaces.

\subsection{Semiconcavity/semiconvexity}

Next we introduce the local semiconcavity/semiconvexity of a geodesic space.
Note that we do not consider a G-space here.

\begin{dfn}\label{dfn:conc}
We say that a geodesic space $X$ is \emph{locally semiconcave} if for every $x\in X$, there exist $\lambda\in\mathbb R$ and a neighborhood $U$ of $x$ such that for any $p\in U$ and any (unit-speed) shortest path $\gamma(t)$ in $U$, the function
\[t\mapsto|p\gamma(t)|^2-\lambda t^2\]
is concave.
\end{dfn}

\begin{dfn}\label{dfn:conv}
We say that a geodesic space $X$ is \emph{locally semiconvex} if for every $x\in X$, there exist $\lambda\in\mathbb R$ and a neighborhood $U$ of $x$ such that for any $p\in U$ and any (unit-speed) shortest path $\gamma(t)$ in $U$, the function
\[t\mapsto|p\gamma(t)|^2+\lambda t^2\]
is convex.
\end{dfn}

In other words, the second derivative of $|p\gamma(t)|^2$ is locally bounded above/below in a generalized sense.
We call the number $\lambda$ the \emph{local semiconcavity/semiconvexity constant} of $X$ on $U$, which is usually assumed to be positive.
In fact, the precise value of $\lambda$ does not play any particular role in this paper.
If $U=X$, we call $X$ \emph{(globally) semiconcave/semiconvex}.

\begin{ex}\label{ex:concconv}
\leavevmode
\begin{itemize}
\item Any Alexandrov space is locally semiconcave.
In particular, any Alexandrov space of nonnegative curvature is globally semiconcave with $\lambda=1$.
\item Any locally CAT space is locally semiconvex (in fact, \emph{strongly} convex in the sense that $\lambda<0$).
In particular, any CAT($0$) space is globally semiconvex with $\lambda=-1$.
\item Any locally Busemann space is locally semiconvex with $\lambda=0$ (see also Remark \ref{rem:sq} below).
\item For Finsler manifolds, the local semiconcavity/semiconvexity constants are described by the flag and tangent curvatures, and the uniform smoothness/convexity constants, see \cite{Oh09uni, Oh21, Sh01}.
\end{itemize}
We refer the reader to \cite{BBI01, AKP24} for Alexandrov spaces, \cite{BH99, AKP19} for CAT spaces, and \cite{Jo97, Pa14} for Busemann spaces.
\end{ex}

\begin{rem}\label{rem:sq}
The local semiconcavity/semiconvexity of the squared distance function $d(p,\cdot)^2$ implies the local semiconcavity/semiconvexity of the non-squared distance function $d(p,\cdot)$ except at $p$, where the latter semiconcavity/semiconvexity constant depends additionally on the distance from $p$.
In fact, for any smooth functions $\varphi,f:\mathbb R\to\mathbb R$, we have
\[(\varphi\circ f)''=\varphi''\circ f\cdot(f')^2+\varphi'\circ f\cdot f''.\]
An argument based on this formula shows that the semiconcavity/semiconvexity of any positive Lipschitz function $f$ is equivalent to the semiconcavity/semiconvexity of its square $f^2$, after an appropriate change of the semiconcavity/semiconvexity constants (which may diverge as $f\to 0,\infty$).
This elementary fact will be used implicitly and frequently throughout the paper.
\end{rem}

\subsection{Differentiability}

We also introduce the local differentiability of a geodesic space.

\begin{dfn}\label{dfn:diff}
We say that a geodesic space $X$ is \emph{locally differentiable} if every point has a neighborhood $U$ such that for any $p\in U$ and any shortest path $\gamma(t)$ in $U$, the function
\[t\mapsto|p\gamma(t)|^2\]
is differentiable.
\end{dfn}

Clearly, it is equivalent to the differentiability of the function $t\mapsto|p\gamma(t)|$ as long as $p\neq\gamma(t)$.
Note that we do not assume the continuity of the derivative.

\subsection{Normal balls}

Finally, we define normal balls for the above spaces.

\begin{dfn}\label{dfn:nor}
Let $X$ be a G-space (resp.\ locally semiconcave/semiconvex space, locally differentiable space).
We say that an open metric ball
\[B\subset X\]
of radius $r>0$ is a \emph{normal ball} if the concentric closed ball $10\bar B$ with radius $10r$ is contained in a neighborhood $U$ of Definition \ref{dfn:g} (resp.\ \ref{dfn:conc}, \ref{dfn:conv}, \ref{dfn:diff}).
\end{dfn}

By convention, if $X$ satisfies two or more conditions, we will consider a normal ball for all those conditions simultaneously: that is, we assume $10\bar B\subset\bigcap_{\alpha}U_\alpha$ for open neighborhoods $U_\alpha$ from the definitions of those conditions.
All arguments in the rest of the paper are carried out in such normal balls.

\section{Angles}\label{sec:ang}

In this section, we introduce angles in locally semiconcave/semiconvex spaces (not necessarily G-spaces) and discuss their basic properties in a parallel manner.
These angles are defined by the first variation formula exactly the same way as in the study of Busemann convex/concave spaces in \cite{FG25,HY25}.

For simplicity, we assume local uniqueness of geodesics, but this is not essential.
We say that a geodesic space $X$ is \emph{locally uniquely geodesic} if every point has a neighborhood $U$ such that every $x,y\in U$ is joined by a unique shortest path in $X$.
A \emph{normal ball} for a locally uniquely geodesic space is defined in the same way as in Definition \ref{dfn:nor}.

For three points $p,x,y$ in a metric space such that $p\neq x\neq y$ (possibly $p=y$), consider the comparison triangle in Euclidean plane with side-lengths $|px|$, $|py|$, and $|xy|$.
We denote by
\[\tilde\angle pxy\]
the angle opposite to the side $|py|$ in this comparison triangle.

\begin{dfn}\label{dfn:ang}
Let $X$ be a locally uniquely geodesic space and $B$ a normal ball.
For any $p,x,y\in B$ such that $p\neq x\neq y$, we define the \emph{angle} $\angle pxy$ by
\[\angle pxy=\lim_{t\to0}\tilde\angle px\gamma(t)\]
if the limit exists, where $\gamma(t)$ is the unique shortest path from $x$ to $y$ parameterized by arc-length.
\end{dfn}

By the law of cosines in Euclidean plane, we have
\begin{align*}
\cos\tilde\angle px\gamma(t)=\frac{(|px|+|p\gamma(t)|)(|px|-|p\gamma(t)|)}{2|px|t}+\frac{t}{2|px|}
\end{align*}
Therefore, Definition \ref{dfn:ang} is equivalent to the validity of the first variation formula:
\begin{equation}\label{eq:fv}
\lim_{t\to0}\frac{|p\gamma(t)|-|px|}{t}=-\cos\angle pxy.
\end{equation}

Note that $\angle pxy\neq\angle yxp$ in general, even if both angles exist (consider a normed plane).
In \cite{FG25,HY25}, this type of angle was called an \emph{angle viewed from a fixed point}, in order to distinguish it from the other type of angle.
However, since we will not use the other angle in this paper, we will not use this long name.

Throughout the paper, we use the following convention.

\begin{conv}
Whenever we consider the angle $\angle pxy$, we assume that $p\neq x\neq y$ as in Definition \ref{dfn:ang} without further mention.
This condition is often omitted.
\end{conv}

\subsection{Semiconcave case}

We first discuss angles in locally semiconcave spaces.
As explained above, we only consider the locally uniquely geodesic case, which is enough for our purpose for studying G-spaces.

Here we list three basic properties of angle that follow from the local semiconcavity.
The detailed proofs can be found in \cite{FG25,HY25}.
For the convenience of the reader (and for later use), we include outlines of the proofs.

\begin{lem}\label{lem:ang}
Let $X$ be a locally semiconcave, locally uniquely geodesic space and $B$ a normal ball.
Then for any $p,x,y\in B$, the angle $\angle pxy$ in Definition \ref{dfn:ang} is well-defined and satisfies the following almost comparison inequality:
\begin{equation}\label{eq:comp}
\tilde\angle px\gamma(t)\le\angle pxy+\delta(t/|px|),
\end{equation}
where $\delta(s)$ is a positive function  such that $\delta(s)\to0$ as $s\to0$, which depends only on the local semiconcavity constant of $X$ on $B$.
\end{lem}

\begin{rem}\label{rem:comp}
The following proof shows that the error function $\delta$ is explicitly given by the sharp inequality
\[\cos\tilde\angle px\gamma(t)\ge\cos\angle pxy-(\lambda-1)t/2|px|,\]
where $\lambda\ge1$ is the local semiconcavity constant.
This explicit form will be used to prove the improved regularity of the Finsler structure in Theorem \ref{thm:conc}.
Note that the error term vanishes in the case of Alexandrov nonnegative curvature, i.e., $\lambda=1$ (see Example \ref{ex:concconv}).
\end{rem}

\begin{proof}[Proof of Lemma \ref{lem:ang}]
The local semiconcavity in Definition \ref{dfn:conc} implies that the function
\[f(t):=\frac{|p\gamma(t)^2|-|px|^2-\lambda t^2}{2|px|t}\]
is nonincreasing in $t>0$.
Hence the limit exists when $t\to 0$.
Furthermore, $f(t)$ is $((\lambda-1)t/2|px|)$-close to the function
\[g(t):=-\cos\tilde\angle px\gamma(t)=\frac{|p\gamma(t)|^2-|px|^2-t^2}{2|px|t}.\]
Thus the limit of $g(t)$ also exists and coincides with that of $f(t)$.
Together with the monotonicity of $f(t)$, this implies the desired inequality \eqref{eq:comp} (cf.\ Remark \ref{rem:comp}).
See \cite[Lemma 4.3(1)]{HY25} for more details (cf.\ \cite[Lemma 4.8]{FG25}).
\end{proof}

\begin{lem}\label{lem:semiconti}
Let $X$ be a locally semiconcave, locally uniquely geodesic space and $B$ a normal ball.
Then the angle is lower semicontinuous in the following sense:
If $p_i,x_i,y_i\in B$ converge to $p,x,y\in B$, respectively, then
\[\liminf_{i\to\infty}\angle p_ix_iy_i\ge\angle pxy\]
\end{lem}

\begin{proof}
Let $\gamma(t)$ (resp.\ $\gamma_i(t)$) denote the shortest path from $x$ to $y$ (resp.\ $x_i$ to $y_i$) parameterized by arc-length.
By definition, $\angle pxy$ is close to $\tilde\angle px\gamma(t)$ for small $t>0$.
For sufficiently large $i$ depending on $t$, $\tilde\angle px\gamma(t)$ is close to $\tilde\angle p_ix_i\gamma_i(t)$.
Applying the almost comparison inequality \eqref{eq:comp} to $\tilde\angle p_ix_i\gamma_i(t)$ and taking the limit as $i\to\infty$ and $t\to0$, we get the desired inequality.
See \cite[Lemma 4.7]{HY25} for more details.
\end{proof}

\begin{lem}\label{lem:sum}
Let $X$ be a locally semiconcave, locally uniquely geodesic space, $B$ a normal ball, and $p\in B$.
Suppose that $y,x,z\in B$ lie on the same shortest path in this order.
Then we have
\[\angle pxy+\angle pxz\le \pi.\]
\end{lem}

\begin{proof}
Let $\gamma(t)$ be the shortest path from $y$ to $z$ such that $\gamma(0)=x$.
By Remark \ref{rem:sq}, the distance function $d(p,\cdot)$ is locally semiconcave around $x$.
Hence its directional derivative satisfies
\[d(p,\gamma(0))^++d(p,\gamma(0))^-\le 0\]
(we use the notation $d(p,\gamma(0))^\pm$ to represent the directional derivative of $d(p,\cdot)$ in the direction $\gamma^\pm(0)$ rather than the right/left derivative).
By the first variation formula \eqref{eq:fv}, this is equivalent to the desired inequality.
See \cite[Lemma 4.8]{HY25} for more details (cf.\ \cite[Lemma 4.12]{FG25}).
\end{proof}

\subsection{Semiconvex case}

Next we discuss angles in locally semiconvex spaces in a parallel manner.
The proofs are the same as the previous ones and thus omitted: it suffices to simply reverse the previous inequalities in an appropriate way.

\begin{lem}\label{lem:ang'}
Let $X$ be a locally semiconvex, locally uniquely geodesic space and $B$ a normal ball.
Then for any $p,x,y\in B$, the angle $\angle pxy$ in Definition \ref{dfn:ang}  is well-defined and satisfies the following almost comparison inequality:
\begin{equation}\label{eq:comp'}
\tilde\angle px\gamma(t)\ge\angle pxy-\delta(t/|px|),
\end{equation}
where $\delta(s)$ is a positive function such that $\delta(s)\to0$ as $s\to0$, which depends only on the local semiconvexity constant of $X$ on $B$.
\end{lem}

Since the explicit form as in Remark \ref{rem:comp} will not be used, we will omit it.

\begin{lem}\label{lem:semiconti'}
Let $X$ be a locally semiconvex, locally uniquely geodesic space and $B$ a normal ball.
Then the angle is upper semicontinuous in the following sense:
If $p_i,x_i,y_i\in B$ converge to $p,x,y\in B$, respectively, then
\[\limsup_{i\to\infty}\angle p_ix_iy_i\le\angle pxy\]
\end{lem}

\begin{lem}\label{lem:sum'}
Let $X$ be a locally semiconvex, locally uniquely geodesic space, $B$ a normal ball, and $p\in B$.
Suppose that $y,x,z\in B$ lie on the same shortest path in this order.
Then we have
\[\angle pxy+\angle pxz\ge \pi.\]
\end{lem}

\subsection{Differentiability revisited}

By using the notion of angle, we can reformulate the local differentiability introduced in Definition \ref{dfn:diff} as follows.
This is due to the first variation formula \eqref{eq:fv}.

\begin{dfn}\label{dfn:diff'}
A locally uniquely geodesic space $X$ is \emph{locally differentiable} if every point has a neighborhood $U$ satisfying the following: the angle is well-defined in $U$, and for any $p\in U$ and $y,x,z\in U$ on a shortest path in this order, we have
\[\angle pxy+\angle pxz=\pi.\]
\end{dfn}

\subsection{General inequalities}

Finally, we prove two general inequalities for angles, both of which follow from the triangle inequality for the distance function.
Note that local semiconcavity/semiconvexity is not assumed here.

\begin{lem}\label{lem:gen1}
Let $X$ be a locally uniquely geodesic space and $B$ a normal ball.
Let $p,q,x,y\in B$ and suppose that $p,x,q$ lie on a shortest path in this order (see Figure \ref{fig:gen1}).
Then we have
\[\angle pxy+\angle qxy\ge\pi,\]
provided that both angles are well-defined.
\end{lem}

\begin{figure}[ht]
\centering
\begin{tikzpicture}
\coordinate[label=below:$p$](p)at(-1,-2);
\coordinate[label=below right:$x$](x)at(0,0);
\coordinate[label=right:$y$](y)at(1,0);
\coordinate[label=above:$q$](q)at(1,2);

\draw(p)to(x)to(q);
\draw(x)to(y);
\draw[dashed](p)to(y);
\draw[dashed](q)to(y);

\fill(p)circle(1.5pt);
\fill(x)circle(1.5pt);
\fill(y)circle(1.5pt);
\fill(q)circle(1.5pt);

\pic[draw, angle radius=5mm] {angle = p--x--y};
\pic[draw, angle radius=5mm] {angle = y--x--q};
\end{tikzpicture}
\caption{}\label{fig:gen1}
\end{figure}

\begin{proof}
Let $\gamma(t)$ be the shortest path from $x$ to $y$ parameterized by arc-length.
By the triangle inequality and the fact that $p,x,q$ are on a shortest path, we have
\[|p\gamma(t)|-|px|+|q\gamma(t)|-|qx|\ge0.\]
Dividing both sides by $t>0$ and taking $t\to0$ as in the first variation formula \eqref{eq:fv}, we obtain the desired inequality (note that $\cos\alpha+\cos\beta\le0$ if and only if $\alpha+\beta\ge\pi$ for any $\alpha,\beta\in[0,\pi]$).
\end{proof}

\begin{lem}\label{lem:gen2}
Let $X$ be a locally uniquely geodesic space and $B$ a normal ball.
Let $p,p',x,y\in B$ and suppose that $p,p',x$ lie on a shortest path in this order (see Figure \ref{fig:gen2}).
Then we have
\[\angle pxy\le\angle p'xy,\]
provided that both angles are well-defined.
\end{lem}

\begin{figure}[ht]
\centering
\begin{tikzpicture}
\coordinate[label=below:$p$](p)at(-1,-3);
\coordinate[label=left:$p'$](p')at(-1/2,-3/2);
\coordinate[label=left:$x$](x)at(0,0);
\coordinate[label=right:$y$](y)at(1,0);

\draw(p)to(x)to(y);
\draw[dashed](p)to(y);
\draw[dashed](p')to(y);

\fill(p)circle(1.5pt);
\fill(x)circle(1.5pt);
\fill(y)circle(1.5pt);
\fill(p')circle(1.5pt);

\pic[draw, angle radius=3mm] {angle = p--x--y};
\end{tikzpicture}
\caption{}\label{fig:gen2}
\end{figure}

\begin{proof}
Let $\gamma(t)$ be the shortest path from $x$ to $y$ parameterized by arc-length.
By the triangle inequality and the fact that $p,p',x$ are on a shortest path, we have
\[|p\gamma(t)|-|px|\le|p'\gamma(t)|-|p'x|\]
Dividing both sides by $t>0$ and taking $t\to0$ as in the first variation formula \eqref{eq:fv}, we obtain the desired inequality.
\end{proof}

\begin{rem}
The above two properties of angle are analogous to those of the \emph{Busemann function}, which is the main ingredient in the proof of the splitting theorem in various contexts.
In fact, what we develop in the next subsection can be viewed as a variant of the splitting argument.
\end{rem}

\section{Semiconcave case}\label{sec:conc}

In this section, we prove Theorem \ref{thm:conc} for locally semiconcave G-spaces.
The proof consists of the following four steps.
In Section \ref{sec:key}, we prove key lemmas regarding angles.
We then introduce strainer coordinates in Section \ref{sec:str} and establish key properties of strainer curves in Section \ref{sec:cur}.
Using these properties, we define the $C^{1,1/2}$ and DC structures in Section \ref{sec:dc}.
Finally, in Section \ref{sec:fin}, we construct a Finsler metric.

\subsection{Key lemmas}\label{sec:key}

Here we prove several key lemmas concerning angles in locally semiconcave G-spaces.

The first property is the local differentiability of locally semiconcave G-spaces.
Although we will prove a stronger claim later in Proposition \ref{prop:unif}, we start with this simpler version, which demonstrates a typical use of semiconcavity and is sufficient to show other properties of angles.

\begin{lem}\label{lem:diff}
Let $X$ be a locally semiconcave G-space, $B$ a normal ball, and $p\in B$.
Suppose that $y,x,z\in B$ lie on the same shortest path in this order.
Then we have
\[\angle pxy+\angle pxz=\pi.\]
In other words, $X$ is locally differentiable (see Definition \ref{dfn:diff'}).
\end{lem}

\begin{proof}
Let $q\in B$ be a point on the extension of the shortest path $px$ beyond $x$.
See Figure \ref{fig:diff}.

\begin{figure}[ht]
\centering
\begin{tikzpicture}
\coordinate[label=below:$p$](p)at(-1,-3);
\coordinate[label=below right:$x$](x)at(0,0);
\coordinate[label=left:$z$](z)at(-1,0);
\coordinate[label=right:$y$](y)at(1,0);
\coordinate[label=above:$q$](q)at(1,3);

\draw(p)to(x)to(q);
\draw(z)to(x)to(y);
\draw[dashed](p)to(y);
\draw[dashed](p)to(z);
\draw[dashed](q)to(y);
\draw[dashed](q)to(z);

\fill(p)circle(1.5pt);
\fill(x)circle(1.5pt);
\fill(y)circle(1.5pt);
\fill(z)circle(1.5pt);
\fill(q)circle(1.5pt);

\pic[draw, angle radius=5mm] {angle = p--x--y};
\pic[draw, angle radius=5mm] {angle = z--x--p};
\end{tikzpicture}
\caption{}\label{fig:diff}
\end{figure}

By Lemma \ref{lem:gen1} for opposite angles, we have
\[\angle pxy+\angle qxy\ge\pi,\quad \angle pxz+\angle qxz\ge\pi.\]
On the other hand, by Lemma \ref{lem:sum} for adjacent angles, we have
\[\angle pxy+\angle pxz\le\pi,\quad\angle qxy+\angle qxz\le\pi.\]
Combining the above four inequalities, we obtain the desired equality.
\end{proof}

\begin{rem}\label{rem:diff1}
The above proof also shows the other equalities, i.e.,
\[\angle pxy+\angle qxy=\pi,\quad\angle pxz+\angle qxz=\pi,\quad\angle qxy+\angle qxz=\pi.\]
\end{rem}

\begin{rem}\label{rem:diff2}
The above argument does not apply to locally semiconvex G-spaces (see Remark \ref{rem:main} for a counterexample).
The reason it works in Lemma \ref{lem:diff} is that Lemma \ref{lem:sum} based on local semiconcavity (together with the geodesic extension property) provides an opposite constraint to Lemma \ref{lem:gen1} derived from the triangle inequality.
On the other hand, Lemma \ref{lem:sum'} based on local semiconvexity gives a restriction in the same direction as the triangle inequality.
In fact, local semiconvexity can be regarded as a variant of the triangle inequality.
\end{rem}

The local differentiability improves the lower semicontinuity of angles in Lemma \ref{lem:semiconti} to the genuine continuity.

\begin{lem}\label{lem:conti}
Let $X$ be a locally semiconcave G-space and $B$ a normal ball.
Then the angle is continuous in the following sense:
If $p_i,x_i,y_i\in B$ converge to $p,x,y\in B$, respectively, then
\[\lim_{i\to\infty}\angle p_ix_iy_i=\angle pxy.\]
\end{lem}

\begin{proof}
Let $z\in B$ (resp.\ $z_i\in B$) be a point on the extension of the shortest path $yx$ (resp.\ $y_ix_i$) beyond $x$ (resp.\ $x_i$).
We may assume that $z_i$ converges to $z$.
By Lemma \ref{lem:diff}, we have
\[\angle pxy+\angle pxz=\pi, \quad \angle p_ix_iy_i+\angle p_ix_iz_i=\pi.\]
On the other hand, by Lemma \ref{lem:semiconti}, we have
\[\liminf_{i\to\infty}\angle p_ix_iy_i\ge\angle pxy,\quad\liminf_{i\to\infty}\angle p_ix_iz_i\ge\angle pxz.\]
Combining these equalities and inequalities yields the desired equality.
\end{proof}

The local differentiability also shows that the angle is well-defined independently of the choice of points on shortest paths.

\begin{lem}\label{lem:well}
Let $X$ be a locally semiconcave G-space and $B$ a normal ball.
Then for any $p,x,y\in B$, we have
\[\angle pxy=\angle p'xy',\]
where $p',y'\in B$ are arbitrary points other than $x$ on (the extension of) the shortest paths $xp$ and $xy$ (beyond $p$ and $y$), respectively.
\end{lem}

\begin{proof}
By the definition of angle in Definition \ref{dfn:ang}, it is clear that $\angle pxy=\angle pxy'$.
Thus it remains to show that $\angle pxy=\angle p'xy$.
We may assume that $p,p',x$ lie on a shortest path in this order.
See Figure \ref{fig:well}.

\begin{figure}[ht]
\centering
\begin{tikzpicture}
\coordinate[label=below:$p$](p)at(-1,-3);
\coordinate[label=left:$p'$](p')at(-1/2,-3/2);
\coordinate[label=above:$x$](x)at(0,0);
\coordinate[label=right:$y$](y)at(1,0);
\coordinate[label=left:$z$](z)at(-1,0);

\draw(p)to(x);
\draw(z)to(x)to(y);
\draw[dashed](p)to(y);
\draw[dashed](p')to(y);
\draw[dashed](p)to(z);
\draw[dashed](p')to(z);

\fill(p)circle(1.5pt);
\fill(x)circle(1.5pt);
\fill(y)circle(1.5pt);
\fill(p')circle(1.5pt);
\fill(z)circle(1.5pt);

\pic[draw, angle radius=3mm] {angle = p--x--y};
\end{tikzpicture}
\caption{}\label{fig:well}
\end{figure}

Let $z\in B$ be a point on the extension of the shortest path $yx$ beyond $x$.
By Lemma \ref{lem:gen2}, we have
\[\angle pxy\le\angle p'xy,\quad\angle pxz\le\angle p'xz.\]
By Lemma \ref{lem:diff}, we have
\[\angle pxy+\angle pxz=\pi,\quad\angle p'xy+\angle p'xz=\pi.\]
Combining these equalities and inequalities yields the desired inequality.
\end{proof}

\begin{rem}\label{rem:well}
Lemma \ref{lem:well} means that the angle is well-defined on the set of shortest paths.
More precisely, let $\Gamma_x$ be the set of shortest paths emanating from $x\in B$.
Then for any $\gamma,\eta\in\Gamma_x$, we can define
\[\angle(\gamma,\eta):=\angle pxy,\]
where $p,y\in B$ are arbitrary points (other than $x$) on $\gamma$ and $\eta$, respectively.
Note that $\angle(\gamma,\eta)\neq\angle(\eta,\gamma)$ in general.
Lemma \ref{lem:diff} and Remark \ref{rem:diff1} also show
\[\angle(\gamma,\eta)+\angle(\gamma,-\eta)=\pi,\quad\angle(\gamma,\eta)+\angle(-\gamma,\eta)=\pi,\]
where $-\gamma,-\eta\in\Gamma_x$ denote the ``opposite'' shortest paths for $\gamma,\eta$, respectively, defined by the unique extendability of geodesics.
Furthermore, by Lemma \ref{lem:conti}, the angle is continuous with respect to the convergence of shortest paths.
\end{rem}

Now we provide a stronger claim, which we call the \emph{uniform convergence of angle}.
This is the most important property in the present paper.
Compare with \cite[Fact (b)]{OS94} (see also \cite[Claim 4.7]{Ot93}) in the Alexandrov setting (not necessarily a G-space), whose proof relies on the Riemannian property of Alexandrov spaces.
We bypass this issue by using the geodesic extension property of a G-space, as in Lemma \ref{lem:diff}.

\begin{prop}\label{prop:unif}
Let $X$ be a locally semiconcave G-space and $B$ a normal ball.
Then for any $p,x,y\in B$, we have
\[|\angle pxy-\tilde \angle pxy|<\delta(|xy|/|px|),\]
where $\delta(s)$ is a positive function independent of $p,x,y$ such that $\delta(s)\to 0$ as $s\to 0$.
More precisely, it is given by the following explicit formula:
\[|\cos\angle pxy-\cos\tilde \angle pxy|<10\lambda|xy|/|px|,\]
where $\lambda\ge1$ is the local semiconcavity constant of $X$ on $B$.
\end{prop}

The point of the above claim is that $\delta$ is independent of $p,x,y$; compare with Definition \ref{dfn:ang} and Lemma \ref{lem:ang}.
As mentioned in Remark \ref{rem:comp}, the explicit formula is necessary for the proof of the $1/2$-H\"older continuity in Theorem \ref{thm:conc}.

\begin{proof}
The proof is a refinement of that of Lemma \ref{lem:diff}.
Let $z\in B$ be a point on the extension of the shortest path $yx$ beyond $x$ such that $|xy|=|xz|$.
Let $q\in B$ be a point on the extension of the shortest path $px$ beyond $x$ such that $|px|=|qx|$.
Set $s:=|xy|/|px|$ and assume $s\ll1$.
See again Figure \ref{fig:diff}.

In what follows, we repeat the same argument as in the proof of Lemma \ref{lem:diff}, but with refined quantitative estimates.
We first provide a comparison angle version of Lemma \ref{lem:gen1}.
By the Euclidean law of cosine, we have
\[\left|\cos\tilde\angle pxy-\frac{|px|-|py|}{|xy|}\right|\le s/2\]
(since $(a^2+b^2-c^2)/2ab=(a-c)/b+(b^2-(a-c)^2)/2ab$).
The same inequality also holds for $\tilde\angle qxy$.
Combining these two inequalities with the triangle inequality $|py|+|qy|\ge|px|+|qx|$, we obtain
\begin{equation}\label{eq:unif1}
\cos\tilde\angle pxy+\cos\tilde\angle qxy\le s.
\end{equation}
Similarly, we have
\begin{equation}\label{eq:unif2}
\cos\tilde\angle pxz+\cos\tilde\angle qxz\le s.
\end{equation}

On the other hand, by the explicit version of the almost comparison inequality in Remark \ref{rem:comp} and Lemma \ref{lem:sum}, we obtain
\begin{equation}\label{eq:unif3}
\cos\tilde\angle pxy+\cos\tilde\angle pxz\ge\cos\angle pxy+\cos\angle pxz-\lambda s\ge-\lambda s.
\end{equation}
Similarly, we have
\begin{equation}\label{eq:unif4}
\cos\tilde\angle qxy+\cos\tilde\angle qxz\ge-\lambda s,
\end{equation}
Combining the inequalities \eqref{eq:unif1}, \eqref{eq:unif2}, \eqref{eq:unif3}, \eqref{eq:unif4} yields the desired inequality.
\end{proof}

Finally, we provide some remark on the work of Pogorelov \cite[Chapter 3]{Po98}.

\begin{rem}\label{rem:pog1}
The properties of angle established in this subsection show that any locally semiconcave G-space satisfies \emph{Axiom A} in \cite[Chapter 3]{Po98}.
Roughly speaking, Axiom A claims the existence of a well-defined angle (called \emph{slope} in \cite{Po98}) on the set of shortest paths, similar to Lemma \ref{lem:well} and Remark \ref{rem:well}, that varies continuously as in Lemma \ref{lem:conti} and satisfies the sum formulae for adjacent/opposite angles as in Lemma \ref{lem:diff} and Remark \ref{rem:diff1}.

More precisely, the uniform convergence of angle, Proposition \ref{prop:unif}, ensures that the slope is certainly equal to minus the cosine of our angle.
Therefore, the main theorem in \cite[Chapter 3]{Po98} that claims the existence of a continuous Finsler metric can be applied to our case.
However, as mentioned in the introduction, since the proof is a little bit sketchy, we shall provide a detailed argument.
For example, although it is not explicitly stated, Axiom A implies the uniform convergence of angle similar to Proposition \ref{prop:unif}.
See Remark \ref{rem:pog3}.
\end{rem}

\subsection{Strainer charts}\label{sec:str}

Next we introduce a strainer chart in a locally semiconcave G-space.
Roughly speaking, a strainer is a collection of points that plays a similar role to a linearly independent family of vectors in the tangent space.

The notion of a strainer was first introduced in the study of Alexandrov spaces by Burago--Gromov--Perelman \cite{BGP92}, then adopted to GCBA spaces by Lytchak--Nagano \cite{LN19}, and more recently, to GNPC spaces by the authors \cite{FG25} (cf.\ \cite{HY25}).
Here we modify the one used in \cite{FG25}, in order to deal with the asymmetry of the angle resulting from the Finsler character of G-spaces.
We remark that the same asymmetry issue was already addressed by Pogorelov \cite{Po90reg, Po98}.

Let $k$ be a positive integer and $\delta$ a small positive number.
An upper bound for $\delta$ will be determined later in Proposition \ref{prop:open} in terms of the integer $k$.

\begin{dfn}\label{dfn:str}
Let $X$ be a locally semiconcave G-space, $B$ a normal ball, and $x\in B$.
A collection of points $p_1,\dots,p_k\in B\setminus\{x\}$ is called a \emph{$(k,\delta)$-strainer} at $x$ if it satisfies
\begin{equation}\label{eq:str}
|\angle p_ixp_j-\pi/2|<\delta
\end{equation}
for all $1\le i< j\le k$.
We call $x$ a \emph{$(k,\delta)$-strained point} and the distance map
\[f:=(d(p_1,\cdot),\dots,d(p_k,\cdot)):B\to\mathbb R^k\]
a \emph{$(k,\delta)$-strainer map} at $x$.
We also call the minimum of $|p_ix|$ the \emph{length} of this strainer.
\end{dfn}

Note that we do not assume any condition on the reversed angles $\angle p_jxp_i$, where $1\le i<j\le k$.
The following two remarks are immediate consequences of the previous lemmas.

\begin{rem}\label{rem:stropen}
By the continuity of angle in Lemma \ref{lem:conti}, if $p_1,\dots,p_k$ is a $(k,\delta)$-strainer at $x$, then it is a $(k,\delta)$-strainer at any point in a neighborhood $U$ of $x$.
In this case, we say that $p_1,\dots,p_k$ is a \emph{$(k,\delta)$-strainer} on $U$, $f$ is a \emph{$(k,\delta)$-strainer map} on $U$, and the minimum of $|p_iU|$ is the \emph{length} of this strainer.
\end{rem}

\begin{rem}\label{rem:strwell}
By the well-definedness of angle in Lemma \ref{lem:well}, if $p_1,\dots,p_k$ is a $(k,\delta)$-strainer at $x$, then $p_1',\dots,p_k'$ is also a $(k,\delta)$-strainer at $x$, where $p_i'\in B$ is an arbitrary point on (the extension of) the shortest path $xp_i$ (beyond $p_i$).
In particular, one can always extend the length of a strainer.
In this sense, a strainer here is a collection of shortest paths through $x$ satisfying the almost orthogonality condition \eqref{eq:str}, rather than a collection of points around $x$.
This is a significant difference from strainers in various settings of singular spaces \cite{BGP92,LN19,FG25,HY25}.
\end{rem}

To state the properties of a strainer map, we recall the following notions.
Let $0<\epsilon\le1$.
We say that a map $f:X\to Y$ between metric spaces is \emph{$\epsilon$-open} if for any $x\in X$ and $y'\in Y$, there exists $x'\in X$ such that
\[f(x')=y',\quad\epsilon|xx'|\le|f(x)f(x')|.\]
We also say that $f$ is \emph{$\epsilon$-bi-Lipschitz} if
\[\epsilon\le\frac{|f(x)f(x')|}{|xx'|}\le\epsilon^{-1}\]
for any distinct $x,x'\in X$.

Now we are ready to state the main results of this subsection.
Here, unless otherwise stated, we consider the Euclidean norm of $\mathbb R^k$.

\begin{prop}\label{prop:open}
Let $X$ be a locally semiconcave G-space and $B$ a normal ball.
Let $f$ be a $(k,\delta)$-strainer map at $x\in B$.
If $\delta\le\delta_k$, then $f$ is an $\epsilon_k$-open map in a neighborhood of $x$, where $\delta_k\ll\epsilon_k$ are positive constants depending only on $k$.
\end{prop}

\begin{prop}\label{prop:bilip}
Let $X$ be a locally semiconcave G-space and $B$ a normal ball.
Let $f$ be a $(k,\delta)$-strainer map at $x\in B$ and suppose that $x$ is not $(k+1,\delta)$-strained.
If $\delta\le\delta_k$, then $f$ is an open $\epsilon_k$-bi-Lipschitz embedding in a neighborhood of $x$, where $\delta_k$ and $\epsilon_k$ are the constants of Proposition \ref{prop:open}.
\end{prop}

The above theorems are analogs of the corresponding results in the setting of Alexandrov/GCBA/GNPC spaces (that are not necessarily G-spaces).
For example, compare with \cite[Proposition 5.17, Theorem 5.30]{FG25} in the GNPC setting.

\begin{rem}
The assumption of Proposition \ref{prop:bilip} is slightly weaker than the corresponding statements in the settings of singular spaces mentioned above.
In those singular settings, one usually needs to assume that \emph{every point} near $x$ is not $(k+1,\delta)$-strained.
This improvement comes from the regularity of a G-space, and is closely related to the constancy of a strainer number discussed below; see Definition \ref{dfn:strnum} and Lemma \ref{lem:strnum}.
\end{rem}

The proof of Proposition \ref{prop:open} is based on the following elementary lemma.

\begin{lem}\label{lem:open}
Let $f:X\to Y$ be a continuous map between metric spaces.
Suppose that $X$ is locally compact, $Y$ is locally geodesic, and there exists $\varepsilon>0$ satisfying the following: for every $x\in X$ and each $v\in Y\setminus\{f(x)\}$ sufficiently close to $f(x)$, there exists $y\in X\setminus\{x\}$ such that
\[|f(y)v|-|f(x)v|\le-\varepsilon|xy|.\]
Then $f$ is an $\varepsilon$-open map.
\end{lem}

For the proof, see \cite[Lemma 5.15]{FG25} (see also \cite{Ly05} for a more general statement).

\begin{proof}[Proof of Proposition \ref{prop:open}]
The proof is essentially the same as in the GNPC case, \cite[Proposition 5.17]{FG25}.
For $v=(v_1,\dots,v_k)\in\mathbb R^k$, we introduce a modified norm
\[||v||:=\sum_{i=1}^k\epsilon^i|v_i|,\]
where $\epsilon$ is a positive constant depending only on $k$, which will be defined below.

By Remark \ref{rem:stropen}, we may assume that $f$ is a $(k,\delta)$-strainer map on a neighborhood $U$ of $x$.
For any $y\in U$ and $v\in\mathbb R^k\setminus\{f(y)\}$ sufficiently close to $f(y)$, we show that there exists $z\in U$ such that
\begin{equation}\label{eq:open}
||f(z)-v||-||f(y)-v||\le-\epsilon'|yz|,
\end{equation}
where $\epsilon'$ is another positive constant depending only on $k$.
Then by Lemma \ref{lem:open}, $f$ is $\epsilon'$-open with respect to $||\cdot||$.
Replacing the modified norm with the Euclidean one, we obtain the claim.

Let us show \eqref{eq:open}.
Since $v\neq f(y)$, there is some $1\le i\le k$ such that $f_i(y)\neq v_i$.
Fix such $i$.
If $f_i(y)>v_i$ (resp.\ $f_i(y)<v_i$), we choose $z$ to be a point sufficiently close to $y$ on the shortest path $p_iy$ (resp.\ on the extension of $p_iy$ beyond $y$) such that $f_i(y)>f_i(z)>v_i$ (resp.\ $f_i(y)<f_i(z)<v_i$).
Then we have
\[
|f_j(z)-v_j|-|f_j(y)-v_j|\le
\begin{cases}
\delta|yz| & j<i \\
-|yz |& j=i \\
|yz| & j>i
\end{cases}.
\]
Here the first inequality follows from the almost orthogonality \eqref{eq:str}, the first variation formula \eqref{eq:fv} (provided $z$ is close to $y$), and the local differentiability as in Lemma \ref{lem:diff} (when $z$ is beyond $y$).
The second inequality follows from the choice of $z$ and the third inequality follows from the $1$-Lipschitz property of $f$.

The above inequalities imply
\begin{align*}
||f(z)-v||-||f(y)-v||&\le\sum_{j=1}^{i-1}\epsilon^j\delta|yz|-\epsilon^i|yz|+\sum_{j=i+1}^k\epsilon^j|yz|\\
&\le-(\epsilon^i/2)|yz|,
\end{align*}
provided $\delta\ll\epsilon$ are small enough compared to $k$.
This completes the proof.
\end{proof}

\begin{rem}\label{rem:open}
The above proof does not need the uniform convergence of angle, Proposition \ref{prop:unif}.
Indeed, in view of the first variation formula \eqref{eq:fv}, it is enough to choose $z$ close to $y$, depending on $y$.
This observation will be important for the proof of Theorem \ref{thm:top} in Section \ref{sec:top}.
\end{rem}

The proof of Proposition \ref{prop:bilip} is based on the following lemma.
Strictly speaking, the one used later is slightly different, but it is written this way for comparison with similar lemmas in various general settings, i.e., Alexandrov/GCBA/GNPC spaces that are not necessarily G-spaces.

\begin{lem}\label{lem:bilip}
Let $X$ be a locally semiconcave G-space and $B$ a normal ball.
Let $p_1,\dots,p_k\in B$ be a $(k,\delta)$-strainer at $x\in B$.
Suppose that there exist $y,z\in B$ sufficiently close to $x$ such that
\[||p_iy|-|p_iz||<\delta'|yz|\]
for all $1\le i\le k$, where $\delta'\ll\delta$ is a positive constant depending only on $\delta$.
Then $p_1,\dots,p_k,y$ is a $(k+1,\delta)$-strainer at $z$.
\end{lem}

In contrast to Remark \ref{rem:open}, the proof of Lemma \ref{lem:bilip} requires the uniform convergence of angle, Proposition \ref{prop:unif}, to ensure the arbitrary choice of $y$ and $z$.
This is a crucial difference from the GNPC case \cite[Proposition 5.7]{FG25}; see also Lemma \ref{lem:bilip'} and Remark \ref{rem:gnpc}.

\begin{proof}[Proof of Lemma \ref{lem:bilip}]
By the openness of a strainer, Remark \ref{rem:stropen}, we may assume that $p_1,\dots,p_k$ is a $(k,\delta)$-strainer at $z$.
Hence it remains to show that
\[|\angle p_izy-\pi/2|<\delta\]
for any $1\le i\le k$.
Proposition \ref{prop:unif} implies that $\angle p_izy$ is $(\delta/2)$-close to $\tilde\angle p_izy$, provided that $y,z$ are sufficiently close to $x$ (depending only on the local semiconcavity constant and the length of the strainer, see Definition \ref{dfn:str}).
Furthermore, the assumption that $||p_iy|-|p_iz||<\delta'|yz|$ implies that $\tilde\angle p_izy$ is $(\delta/2)$-close to $\pi/2$, provided $\delta'\ll\delta$.
This completes the proof.
\end{proof}

\begin{proof}[Proof of Proposition \ref{prop:bilip}]
Since $f$ consists of distance functions, it is $\sqrt k$-Lipschitz (with respect to the Euclidean norm).
By Proposition \ref{prop:open}, $f$ is $\epsilon_k$-open.
Hence it remains to show that $f$ is locally injective around $x$ (see \cite[Proof of Theorem 5.30]{FG25} for more details).

Suppose that $f$ is not locally injective around $x$.
Then there exist $y_j,z_j\in B$ converging to $x$ such that $f(y_j)=f(z_j)$.
By the proof of Lemma \ref{lem:bilip}, we see that
\[\lim_{j\to\infty}\angle p_iz_jy_j=\pi/2\]
for any $1\le i\le k$.
By the extension property of a strainer in Remark \ref{rem:strwell}, we may replace $y_j$ by $y_j'$ which has the same property as above and whose distance to $z_j$ is uniformly bounded below by a positive number.
Passing to a subsequence and using the continuity of angle in Lemma \ref{lem:conti}, we may assume that $y_j'$ converges to some $p_{k+1}\neq x$ such that $\angle p_ixp_{k+1}=\pi/2$ for any $1\le i\le k$.
In particular, $p_1,\dots,p_k,p_{k+1}$ is a $(k+1,\delta)$-strainer at $x$, which contradicts the assumption.
\end{proof}

We will use Proposition \ref{prop:bilip} to introduce an atlas for a locally semiconcave G-space.
To this end, we need the following definition and lemma.
By a \emph{$(k,0)$-strainer}, we mean a $(k,\delta)$-strainer for all $\delta>0$, i.e., $\angle p_ixp_j=\pi/2$ instead of \eqref{eq:str}.

\begin{dfn}\label{dfn:strnum}
Let $X$ be a locally semiconcave G-space.
A \emph{local strainer number} at $x\in X$ is the supremum of $k$ such that $x$ is a $(k,0)$-strained point.
\end{dfn}

\begin{lem}\label{lem:strnum}
The local strainer number of a locally semiconcave G-space $X$ is finite and constant.
In other words, there exists an integer $n$ such that every point of $X$ is $(n,0)$-strained.
\end{lem}

Before the proof, one more remark.

\begin{rem}\label{rem:bilip}
With the above terminology, Proposition \ref{prop:bilip} is valid for $\delta=0$.
This directly follows from the above proof.
The reason we do not primarily use a $(k,0)$-strainer is because it is not an open condition.
\end{rem}

\begin{proof}[Proof of Lemma \ref{lem:strnum}]
We first show the finiteness.
Suppose that there exists $x\in X$ for which the local strainer number is infinite.
By definition, for any positive integer $k$, there exists a $(k,0)$-strainer $p_1^k,\dots,p_k^k$ at $x$.
By the extension property of a strainer in Remark \ref{rem:strwell}, we may assume that $|p_i^kx|$ is uniformly bounded below by a positive number.
Using a diagonal argument and the continuity of angle in Lemma \ref{lem:conti}, one can find an infinite sequence of points $p_1,p_2,\ldots\in B\setminus\{x\}$ such that
\[\angle p_ixp_j=\pi/2\]
for all $i<j$.
Passing to a subsequence, we may assume that $p_i$ converges to $p_\infty\neq x$.
However, taking $i,j\to\infty$ in the above equality gives $\angle p_\infty xp_\infty=\pi/2$, which is a contradiction.

Next we show the constancy.
Suppose that the local strainer number at $x\in X$ is $k$.
Then by Proposition \ref{prop:bilip} and Remark \ref{rem:bilip}, there exists a neighborhood of $x$ that is homeomorphic to $\mathbb R^k$.
Since $X$ is connected, such $k$ must be unique.
\end{proof}

Let $X$ be a locally semiconcave G-space.
Lemma \ref{lem:strnum} allows us to define the \emph{strainer number} of $X$ by the local strainer number at an arbitrary point of $X$.

\begin{dfn}\label{dfn:strch}
Let $X$ be a locally semiconcave G-space with strainer number $n$.
A \emph{strainer chart} around $x\in X$ is an $(n,\delta_n)$-strainer map on an open neighborhood of $x$ that is an open bi-Lipschitz embedding to $\mathbb R^n$ as in  Proposition \ref{prop:bilip}.
\end{dfn}

We conclude this subsection with the following summary.

\begin{cor}\label{cor:str}
Let $X$ be a locally semiconcave G-space with strainer number $n$.
Then every point of $X$ admits a strainer chart of dimension $n$.
In particular, $X$ is a Lipschitz manifold of dimension $n$.
\end{cor}

\subsection{Strainer curves}\label{sec:cur}

Here we establish important technical results, which will be used to introduce a $C^{1,1/2}$ and DC atlas and a Finsler metric in the next two subsections.
Since it is unclear whether our G-space has a nice tangent cone as in the Alexandrov/CAT case (or more generally the Busemann concave/convex case), its ``tangent vector'' requires a slightly careful treatment.

We begin with the definition of a strainer curve in a locally semiconcave G-space.
In fact, it is just a pull-back of an affine geodesic by a strainer chart.

\begin{dfn}\label{dfn:cur}
Let $X$ be a locally semiconcave G-space and $B$ a normal ball.
Let $f:U\to\mathbb R^n$ be a strainer chart as in Definition \ref{dfn:strch}, where $U\subset B$ is an open subset.
For any $x\in U$ and $v\in\mathbb R^n$ ($=T_{f(x)}\mathbb R^n$), we set
\[c(t):=f^{-1}(f(x)+tv)\]
and call it a \emph{strainer curve} passing through $x$ in the direction $v$ (or more precisely, in the direction ``$df^{-1}(v)$'', which is not yet defined).
\end{dfn}

This subsection is devoted to the study of the basic properties of a strainer curve.
We prove that any strainer curve has well-defined speed and tangent direction in some suitable sense.

The following lemma is the key result.
The proof is an application of the uniform convergence of angle, Proposition \ref{prop:unif}, and the local differentiability, Lemma \ref{lem:diff}.

\begin{lem}\label{lem:cur1}
Let $c(t)=f^{-1}(f(x)+tv)$ be a strainer curve as in Definition \ref{dfn:cur}.
For any $0<|s|<t\ll1$, let $c_s(t)$ denote a point on the extension of the shortest path $xc(s)$ that lies on the same side with $c(t)$ (defined by the sign of $s$) such that
\[|x,c_s(t)|=(t/|s|)|x,c(s)|\]
(see Figures \ref{fig:cur1} and \ref{fig:cur2}).
Then we have
\[|c(t),c_s(t)|<\delta(t)t|v|,\]
where $\delta(t)=Ct|v|$ and $C$ is a positive constant depending only on the local semiconcavity constant on $U$, the strainer number of $X$, and the strainer length of $f$ (see Definition \ref{dfn:str}).
\end{lem}

\begin{figure}[ht]
\centering
\begin{tikzpicture}
\draw (0,0) .. controls (2,0.5) and (4,0.5) .. (6,0)
    coordinate[pos=4/11] (x)
    coordinate[pos=7/11] (cs)
    coordinate[pos=10/11] (ct);

\fill (x) circle (1.5pt) node[below] {$x$};
\fill (cs) circle (1.5pt) node[below] {$c(s)$};
\fill (ct) circle (1.5pt) node[below] {$c(t)$};

\coordinate[label=above:$c_s(t)$](ct')at($(x)!2!(cs)$);
\fill (ct') circle (1.5pt);

\draw[dashed](x)to(ct');
\draw[dotted](ct)to(ct');
\end{tikzpicture}
\caption{The $s>0$ case}\label{fig:cur1}
\end{figure}

\begin{figure}[ht]
\centering
\begin{tikzpicture}
\draw (0,0) .. controls (2,0.5) and (4,0.5) .. (6,0)
    coordinate[pos=1/11] (cs)
    coordinate[pos=4/11] (x)
    coordinate[pos=10/11] (ct);

\fill (x) circle (1.5pt) node[below] {$x$};
\fill (cs) circle (1.5pt) node[below] {$c(s)$};
\fill (ct) circle (1.5pt) node[below] {$c(t)$};

\coordinate[label=above:$c_s(t)$](ct')at($(cs)!3!(x)$);
\fill (ct') circle (1.5pt);

\draw[dashed](x)to(ct');
\draw[dotted](ct)to(ct');
\end{tikzpicture}
\caption{The $s<0$ case}\label{fig:cur2}
\end{figure}

As mentioned earlier, the explicit form of the error function will be used in the proof of the $1/2$-H\"older continuity of the Finsler structure in Theorem \ref{thm:conc}.
On the other hand, the implicit form will be reused later in the proof of Theorem \ref{thm:conv} in Section \ref{sec:conv}.
This is why we stick to the implicit form throughout this subsection.

\begin{proof}[Proof of Lemma \ref{lem:cur1}]
In what follows, we abuse the symbol $\delta(t)$ to denote various linear functions $Ct|v|$, whose coefficients $C>0$ satisfy the dependence described in the statement of Lemma \ref{lem:cur1}.

We first consider the $s>0$ case (see Figure \ref{fig:cur1}).
By definition
\[f(c(t))=f(x)+tv,\quad f(c(s))=f(x)+sv,\]
which implies
\begin{equation}\label{eq:cur1-1}
(t/s)(f(c(s))-f(x))-(f(c(t))-f(x))=0.
\end{equation}
Suppose $f=(d(p_1,\cdot),\dots,d(p_n,\cdot))$.
Since $c_s(t)$ is on the extension of the shortest path $xc(s)$, by applying the uniform convergence of angle in Proposition \ref{prop:unif} (the explicit version) twice, we have
\begin{equation}\label{eq:cur1-2}
\left|\frac{|p_i,c(s)|-|p_ix|}{|x,c(s)|}-\frac{|p_i,c_s(t)|-|p_ix|}{|x,c_s(t)|}\right|<100\lambda\frac{|x,c_s(t)|}{|p_ix|},
\end{equation}
where $\lambda\ge1$ is the local semiconcavity constant on $U$.
Since $f$ is $\epsilon_n$-bi-Lipschitz by Proposition \ref{prop:bilip} and $s>0$, we have
\begin{equation}\label{eq:cur1-3}
|x,c_s(t)|=(t/s)|xc(s)|\le\epsilon_n^{-1}t|v|.
\end{equation}
Combining \eqref{eq:cur1-2} and \eqref{eq:cur1-3}, we get
\begin{equation}\label{eq:cur1-4}
\left|\frac{f(c(s))-f(x)}{|x,c(s)|}-\frac{f(c_s(t))-f(x)}{|x,c_s(t)|}\right|<\delta(t),
\end{equation}
where $\delta(t)=Ct|v|$ as in the statement of Lemma \ref{lem:cur1}.
Combining \eqref{eq:cur1-3} with \eqref{eq:cur1-4} again, we have
\begin{equation}\label{eq:cur1-5}
\left|(t/s)(f(c(s))-f(x))-(f(c_s(t))-f(x))\right|<\delta(t)t|v|.
\end{equation}
Combining \eqref{eq:cur1-1} and \eqref{eq:cur1-5}, we obtain
\[|f(c(t))-f(c_s(t))|<\delta(t)t|v|.\]
Since $f$ is $\epsilon_n$-bi-Lipschitz, this shows the claim.

Next we consider the $s<0$ case (see Figure \ref{fig:cur2}).
The equality \eqref{eq:cur1-1} remains true.
Instead of \eqref{eq:cur1-4}, additionally using the local differentiability, Lemma \ref{lem:diff}, we get
\[\left|-\frac{f(c(s))-f(x)}{|x,c(s)|}-\frac{f(c_s(t))-f(x)}{|x,c_s(t)|}\right|<\delta(t).\]
This, together with $|x,c_s(t)|=(-t/s)|x,c(s)|$, implies \eqref{eq:cur1-5}.
The desired inequality now follows from \eqref{eq:cur1-1} and \eqref{eq:cur1-5} as before.
This completes the proof.
\end{proof}

Using the above lemma, we first show that every strainer curve has a well-defined speed.

\begin{lem}\label{lem:cur2}
Let $c(t)=f^{-1}(f(x)+tv)$ be a strainer curve as in Definition \ref{dfn:cur}.
Then the limit $|\dot c(0)|:=\lim_{t\to0}|x,c(t)|/|t|$ exists and further satisfies
\[\left||\dot c(0)|-\frac{|x,c(t)|}{|t|}\right|<\delta(t)|v|\]
for any small $t$ (possibly negative), where $\delta(t)=C|t||v|$ as in Lemma \ref{lem:cur1}.
In particular, $|\dot c(0)|$ depends continuously on $x$ and $v$.
\end{lem}

\begin{proof}
Let $0<|s|<t\ll1$.
By using the same notation as in Lemma \ref{lem:cur1},
\[\frac{|x,c(s)|}{|s|}=\frac{|x,c_s(t)|}{t}.\]
Therefore, by Lemma \ref{lem:cur1},
\begin{align*}
\left|\frac{|x,c(s)|}{|s|}-\frac{|x,c(t)|}{t}\right|&=\left|\frac{|x,c_s(t)|}{t}-\frac{|x,c(t)|}{t}\right|\\
&\le\frac{|c(t),c_s(t)|}{t}<\delta(t)|v|,
\end{align*}
which shows the claim.
Finally, since a uniform limit of continuous functions is continuous, the value $|\dot c(0)|$ depends continuously on $x$ and $v$.
\end{proof}

\begin{rem}\label{rem:cur2}
By definition, we have
\[\epsilon_n|v|\le|\dot c(0)|\le\epsilon_n^{-1}|v|,\]
where $\epsilon_n$ is the bi-Lipschitz constant of $f$ from Proposition \ref{prop:bilip}.
\end{rem}

We next show that every strainer curve has a well-defined tangent direction in some suitable sense.

\begin{lem}\label{lem:cur3}
Let $c(t)=f^{-1}(f(x)+tv)$ be a strainer curve as in Definition \ref{dfn:cur}.
For any sequence $t_i\to 0$ of positive numbers, suppose that the unique extension of the shortest path $xc(t_i)$ (beyond both $x$ and $c(t_i)$) converges to a (unit-speed) shortest path $\gamma$ such that $\gamma(0)=x$.
Then we have
\[|\gamma(|\dot c(0)|t),c(t)|<\delta(t)|t||v|\]
for any small $t$ (possibly negative), where $\delta(t)=C|t||v|$ as in Lemma \ref{lem:cur1}.

In particular, for any directionally differentiable, locally Lipschitz function $g$ around $x$, we have
\[(g\circ c)^+(0)=|\dot c(0)|(g\circ\gamma)^+(0),\quad (g\circ c)^-(0)=|\dot c(0)|(g\circ\gamma)^-(0).\]
Here the directional differentiability means that the right-hand sides of the above equalities are well-defined.
\end{lem}

\begin{proof}
The first claim is an immediate consequence of Lemmas \ref{lem:cur1} and \ref{lem:cur2}.
Indeed, by the definition of $\gamma$ and Lemma \ref{lem:cur2}, $\gamma(|\dot c(0)|t)$ is the limit of $c_{t_i}(t)$, which in turn is $\delta(t)|t||v|$-close to $c(t)$ by Lemma \ref{lem:cur1}.

The second claim follows from the first.
Suppose that $t>0$ and $g$ is $L$-Lipschitz around $x$.
By the first claim, we have
\begin{equation}\label{eq:cur3}
\left|\frac{g(c(t))-g(x)}t-\frac{g(\gamma(|\dot c(0)|t))-g(x)}t\right|\le\frac{L|\gamma(|\dot c(0)|t),c(t)|}{t}<L\delta(t)|v|.
\end{equation}
Taking $t\to0$, we obtain $(g\circ c)^+(0)=|\dot c(0)|(g\circ\gamma)^+(0)$.
The other equality follows in the same way.
\end{proof}

In particular, we obtain the continuous differentiability of a distance function with respect to a strainer curve.

\begin{lem}\label{lem:cur4}
Let $c(t)=f^{-1}(f(x)+tv)$ be a strainer curve as in Definition \ref{dfn:cur}.
For any $p\in\tilde B\setminus\{x\}$, set $g:=d(p,\cdot)$, where $\tilde B$ is another normal ball that contains $x$.
Then the composition $g\circ c$ is differentiable at $x$ and satisfies
\[\left|\frac{g(c(t))-g(x)}{t}-(g\circ c)'(0)\right|<\tilde\delta(t)|v|,\]
for any small $t$ (possibly negative), where $\tilde \delta(t)=\tilde C|t||v|$ as in Lemma \ref{lem:cur1}, which additionally depends on the local semiconcavity constant on $\tilde B$ and the distance from $p$ to $x$.
In particular, $(g\circ c)'(0)$ depends continuously on $x$ and $v$.
\end{lem}

\begin{proof}
Let $\gamma(t)$ be a (unit-speed) tangential shortest path for $c$ as in Lemma \ref{lem:cur3}.
Since $g$ is a distance function, by the uniform convergence of angle in Proposition \ref{prop:unif} (the explicit version), we have
\begin{equation}\label{eq:cur4}
\left|\frac{g(\gamma(t))-g(x)}{t}-(g\circ\gamma)'(0)\right|<100\tilde\lambda\frac{t}{|px|},
\end{equation}
where $\tilde\lambda\ge 1$ is the local semiconcavity constant on $\tilde B$.
Combining \eqref{eq:cur4} with \eqref{eq:cur3} and Remark \ref{rem:cur2} shows the claim.
The right-hand side of \eqref{eq:cur4} explains the additional dependency of $\tilde\delta$ on $\tilde\lambda$ and $|px|$.
More precisely, $\tilde\delta$ depends on a lower bound for $|px|$, and thus it can be chosen uniformly for every $x'$ close to $x$.
In particular, the convergence $\lim_{t\to0}(g(c(t))-g(x))/t=(g\circ c)'(0)$ is locally uniform, and hence the limit depends continuously on $x$ and $v$.
\end{proof}

\begin{rem}
Unfortunately, it is still unclear whether the tangential shortest path $\gamma$ in Lemma \ref{lem:cur3} is unique (in view of the first claim of Lemma \ref{lem:cur3}, the uniqueness holds, for example, if $X$ further satisfies the Busemann concavity as in \cite{HY25}).
However, Lemma \ref{lem:cur4} suggests that the differential ``$\dot c(0)=df^{-1}(v)$'' is well-defined as a derivation at $x$ that acts on the class of distance functions.
\end{rem}

\subsection{Differentiable and DC structures}\label{sec:dc}

We now prove that the strainer charts of a locally semiconcave G-space define canonical $C^{1,1/2}$ and DC structures.

We first discuss the $C^{1,1/2}$ structure.
Recall that an atlas on a topological manifold is called \emph{$C^{1,1/2}$} if the coordinate transformations are $C^{1,1/2}$ diffeomorphisms.
The following proposition is the first main result of this subsection.

\begin{prop}\label{prop:c1}
Let $X$ be a locally semiconcave G-space.
Then any coordinate transformation of strainer charts is a $C^{1,1/2}$ diffeomorphism.
More precisely, for any $p\in X$, the distance function $d(p,\cdot)$ is $C^{1,1/2}$ smooth (except at $p$) with respect to any strainer chart that is contained in a normal ball around $p$.
\end{prop}

The proof is an easy application of Lemma \ref{lem:cur4} and the explicit error function there.
In fact, the explicit form will be used only in the final line of the proof.

\begin{proof}
Fix $p\in X$ and set $g:=d(p,\cdot)$.
Let $f:U\to\mathbb R^n$ be a strainer chart on $X$ that is contained in a normal ball around $p$.
We show that $g\circ f^{-1}$ is a $C^{1,1/2}$ function on $f(U)\setminus\{f(p)\}$.
Note that Lemma \ref{lem:cur4} already showed that it is a $C^1$ function.
More precisely, we observed that the differential of $g\circ f^{-1}$ is a uniform limit of differential quotients, which are continuous functions.
In what follows, we improve it to the local $1/2$-H\"older continuity, by examining the previous proof in detail and making use of the explicit error function.

Let $c_x(t)=f^{-1}(f(x)+tv)$ and $c_y(t)=f^{-1}(f(y)+tv)$ be strainer curves passing through $x,y\in U\setminus\{p\}$ in the same direction $v\in\mathbb R^n$, respectively.
By Lemma \ref{lem:cur4}, we have
\begin{equation}\label{eq:c1-1}
\left|\frac{g(c_x(t))-g(x)}{t}-(g\circ c_x)'(0)\right|<\tilde\delta(t)|v|,\quad\left|\frac{g(c_y(t))-g(y)}{t}-(g\circ c_y)'(0))\right|<\tilde\delta(t)|v|.
\end{equation}
Note that the error function $\tilde\delta(t)$ can be chosen uniformly for any close $x$ and $y$, for the reason described in the proof of Lemma \ref{lem:cur4}.
By the $1$-Lipschitz property of $g$ and the $\epsilon_n$-bi-Lipschitz property of $f$ (Proposition \ref{prop:bilip}), we have
\begin{equation}\label{eq:c1-2}
\begin{aligned}
\left|\frac{g(c_x(t))-g(x)}{t}-\frac{g(c_y(t))-g(y)}{t}\right|&\le\frac1{|t|}\left(|g(c_x(t))-g(c_y(t))|+|g(x)-g(y)|\right)\\
&\le\frac1{|t|}(|(c_x(t),c_y(t)|+|xy|)\\
&\le\frac{2\epsilon_n^{-2}}{|t|}|xy|.
\end{aligned}
\end{equation}
Combining \eqref{eq:c1-1} and \eqref{eq:c1-2}, we obtain
\[|(g\circ c_x)'(0)-(g\circ c_y)'(0)|<\frac{2\epsilon_n^{-2}}{|t|}|xy|+2\tilde\delta(t)|v|.\]
Since $\tilde\delta(t)=\tilde C|t||v|$ by Lemma \ref{lem:cur4}, letting $t:=|xy|^{1/2}$ shows the claim.
\end{proof}

Next we discuss the DC structure.
The DC structure was first introduced by Perelman \cite{Per94} in the setting of Alexandrov spaces and later used by Lytchak--Nagano \cite{LN19} in the study of GCBA spaces.
See also a general theory by Ambrosio--Bertrand \cite{AB18}.
Although it is possible to define the DC structure in the same way as the $C^{1,1/2}$ structure (i.e., the coordinate transformations are DC isomorphisms between Euclidean domains), here we follow a slightly more general treatment in \cite[Section 14]{LN19}.

Let $X$ be a geodesic space and $U$ an open subset.
A locally Lipschitz function $f:U\to\mathbb R$ is called \emph{semiconcave} if for any $x\in U$ there exists $\lambda>0$ such that for any (unit-speed) shortest path $\gamma(t)$ in $B(x,\lambda^{-1})$, the function
\[f\circ\gamma(t)-\lambda t^2\]
is concave. 
We say that $f$ is \emph{DC} if it is locally represented as a difference of two semiconcave functions.

\begin{dfn}
Let $X$ and $Y$ be geodesic spaces.
A locally Lipschitz map
\[f:X\to Y\]
is a \emph{DC map} if for any DC function $g:V\to \mathbb R$, where $V\subset Y$ is open, the composition $g\circ f$ is a DC function on $f^{-1}(V)$.
We also say that $f$ is a \emph{DC isomorphism} if it is a locally bi-Lipschitz homeomorphism and $f$ and $f^{-1}$ are DC maps.
\end{dfn}

A composition of DC maps is DC.
In case $Y=\mathbb R^n$, the above definition recovers the component-wise definition, i.e., $f$ is DC if and only if each component of $f$ is a DC function in the previous sense.
See \cite[Section 14]{LN19} for more details.

Now we are ready to state the other main result of this subsection.

\begin{prop}\label{prop:dc}
Let $X$ be a locally semiconcave G-space.
Then any strainer coordinate map on $X$ is a DC isomorphism.
In particular, any coordinate transformation of strainer charts is a DC isomorphism.
\end{prop}

The proof of Proposition \ref{prop:dc} is essentially the same as in \cite[Section 3]{Per94}/\cite[Section 14.6]{LN19} for Alexandrov/GCBA spaces.
However, since several minor modifications are required, we include a complete argument.

Let $X$ be a locally semiconcave G-space and $f:U\to\mathbb R^n$ a strainer chart on $X$.
Modifying the definition in \cite[Section 14.6]{LN19}, we introduce the following notion: for given $\epsilon,\delta>0$, we say that a semiconcave function $g$ on $U$ is \emph{$(\epsilon,\delta)$-special} for $f$ if for any shortest path $\gamma$ with $\gamma(0)\in U$, we have
\begin{equation}\label{eq:sp}
(f_i\circ\gamma)^+(0)\le\delta\text{ for all }1\le i\le n\implies (g\circ\gamma)^+(0)\ge\epsilon.
\end{equation}

The proof of Proposition \ref{prop:dc} is based on the following two lemmas, essentially the same as \cite[Lemmas 14.5, 14.6]{LN19}.

\begin{lem}\label{lem:dc1}
Let $X$ be a locally semiconcave G-space and $f:U\to\mathbb R^n$ a strainer chart on $X$.
Then, for any $x\in U$, there exists a $1$-Lipschitz, $(\epsilon_n,\delta_n)$-special semiconcave function $g$ for $f$ defined on a neighborhood of $x$, where $\epsilon_n\gg\delta_n$ are positive constants depending only on $n$.
\end{lem}

\begin{lem}\label{lem:dc2}
Let $X$ be a locally semiconcave G-space and $f:U\to\mathbb R^n$ a strainer chart on $X$.
Let $g$ be an $(\epsilon,\delta)$-special semiconcave function for $f$ defined on $U$, where $\epsilon$ and $\delta$ are positive numbers (possibly depending on $g$).
Then the composition $g\circ f^{-1}$ is a semiconcave function on $f(U)$.
\end{lem}

\begin{proof}[Proof of Lemma \ref{lem:dc1}]
The proof is the same as the original one, \cite[Lemma 14.5]{LN19}, using the local differentiability proved in Lemma \ref{lem:diff}.

Let $p_1,\dots,p_n\in B$ be the $(n,\delta_n')$-strainer that defines $f$, where $B$ is a normal ball and $\delta_n'$ is the constant of Proposition \ref{prop:open}.
Fixing $x\in U$, we consider an opposite strainer $q_1,\dots,q_n\in B$, that is, $q_i$ is a point on the extension of a shortest path $p_ix$ beyond $x$.
By the local differentiability, Lemma \ref{lem:diff} and Remark \ref{rem:diff1}, $q_1,\dots,q_n$ is indeed an $(n,\delta_n')$-strainer at $x$.
We show that for a sufficiently small neighborhood $V\subset U$ of $x$,
\[g:=\frac1n\sum_{i=1}^n g_i,\quad g_i:=d(q_i,\cdot)\]
is the desired special function on $V$ (clearly $g$ is $1$-Lipschitz).

Let $\gamma$ be a shortest path such that $\gamma(0)\in V$.
Suppose $(f_i\circ\gamma)^+(0)\le \delta$ for all $1\le i\le n$, where $\delta=\delta_n$ will be determined later.
Then we have
\begin{equation}\label{eq:dc1-1}
(g_i\circ\gamma)^+(0)\ge-2\delta,
\end{equation}
provided $V$ is sufficiently small (indeed, by the local differentiability, Lemma \ref{lem:diff} and Remark \ref{rem:diff1}, and the continuity of angle, Lemma \ref{lem:conti}, the sum of $(f_i\circ\gamma)^+(0)$ and $(g_i\circ\gamma)^+(0)$ must be almost zero). 
On the other hand, by Proposition \ref{prop:bilip}, the strainer map $(g_1,\dots,g_n)$ is $\epsilon$-bi-Lipschitz on $V$ for some $\epsilon$ depending only on $n$.
This implies that $\sum_{i=1}^n|(g_i\circ\gamma)^+(0)|\ge\epsilon$, and hence there exists $1\le i_0\le n$ such that
\begin{equation}\label{eq:dc1-2}
|(g_{i_0}\circ\gamma)^+(0)|\ge\epsilon/n.
\end{equation}
Choosing $\delta\ll\epsilon$ and combining \eqref{eq:dc1-1} and \eqref{eq:dc1-2}, we obtain the desired property \eqref{eq:sp} for some suitable $\epsilon_n$ and $\delta_n$.
\end{proof}

\begin{proof}[Proof of Lemma \ref{lem:dc2}]
The proof is almost identical to the original one, \cite[Section 3, Proposition (A)]{Per94}, with minor modifications mainly based on Lemma \ref{lem:cur3}.

Let $f$ and $g$ be as in the assumption of Lemma \ref{lem:dc2}.
Let $c(t):=f^{-1}(f(x)+tv)$ be a strainer curve with $|v|=1$, as in Definition \ref{dfn:cur}.
To prove that $g\circ f^{-1}$ is semiconcave, it suffices to show that
\begin{gather}
(g\circ c)^+(0)+(g\circ c)^-(0)\le 0,\label{eq:dc2-1}\\
(g\circ c)(t)\le (g\circ c)(0)+(g\circ c)^+(0)t+Ct^2,\label{eq:dc2-2}
\end{gather}
for any small $t>0$, where $C$ is a uniform positive constant independent of $x,t,v$.
The other inequality for $t<0$ and $(g\circ c)^{-}(0)$ similar to \eqref{eq:dc2-2} follows by replacing $v$ with $-v$.
Note that our notation $(g\circ c)^\pm(0)$ represents the directional derivative of $g$ in the direction $c^\pm(0)$ rather than the right/left derivative.

Let $\gamma (t)$ be a (unit-speed) tangential shortest path for $c$ as in Lemma \ref{lem:cur3}.
Then the first inequality \eqref{eq:dc2-1} follows from the second statement of Lemma \ref{lem:cur3}, since $g$ is semiconcave and hence \eqref{eq:dc2-1} with $c$ replaced by $\gamma$ holds (notice that this part is different from the original proof in \cite{Per94}).

Let us prove the second inequality \eqref{eq:dc2-2} in the same way as the original proof.
In what follows we denote by $C$ various (large) positive constants depending only on $n$, the local semiconcavity constant on $U$, the strainer length of $f$, and the semiconcavity of $g$.
Let $f_i$ denote the $i$-th component of $f$, where $1\le i\le n$.
By the semiconcavity of $f_i$ and $g$ and Remark \ref{rem:cur2}, we have
\begin{gather}
(f_i\circ\gamma)(|\dot c(0)|t)\le f_i(x)+|\dot c(0)|(f_i\circ\gamma)^+(0)t+Ct^2,\label{eq:dc2-3}\\
(g\circ\gamma)(|\dot c(0)|t)\le g(x)+|\dot c(0)|(g\circ\gamma)^+(0)t+C t^2.\label{eq:dc2-4}
\end{gather}
By definition, $f_i\circ c(t)=f_i(x)+tv_i$.
This, together with Lemma \ref{lem:cur3} and \eqref{eq:dc2-3}, implies that
\[(f_i\circ\gamma)(|\dot c(0)|t)\le(f_i\circ c)(t)+Ct^2.\]
By the $\epsilon_n$-openness of $f$, Proposition \ref{prop:open}, we can find $y\in U$ such that
\begin{equation}\label{eq:dc2-5}
f_i(y)\le(f_i\circ c)(t),\quad|y,\gamma(|\dot c(0)|t)|\le Ct^2
\end{equation}
for all $1\le i\le n$ (more specifically, we choose $y$ near $\gamma(|\dot c(0)|t)$ such that $f_i(y)=\min\{(f_i\circ c)(t),(f_i\circ\gamma)(|\dot c(0)|t)\}$).
Combining the second inequality of \eqref{eq:dc2-5} with \eqref{eq:dc2-4} by the $1$-Lipschitz property of $g$, we have
\begin{equation}\label{eq:dc2-6}
g(y)\le g(x)+(g\circ c)^+(0)t+Ct^2,
\end{equation}
where we also used Lemma \ref{lem:cur3}.
Moreover, as we will see below, the first inequality of \eqref{eq:dc2-5} implies
\begin{equation}\label{eq:dc2-7}
(g\circ c)(t)\le g(y).
\end{equation}
Combining \eqref{eq:dc2-6} and \eqref{eq:dc2-7} proves \eqref{eq:dc2-2}.

It remains to prove \eqref{eq:dc2-7}.
For fixed $t$, consider a nonempty compact subset
\begin{equation*}
A:=\left\{z\in U\;\middle|\;
\begin{gathered}
f_i(y)\le f_i(z)\le (f_i\circ c)(t)\\
(g\circ c)(t)\le g(z)
\end{gathered}
\right\},
\end{equation*}
where $i$ runs over all integers from $1$ to $n$.
The nonemptiness of $A$ follows from the first inequality of \eqref{eq:dc2-5}, taking $z=c(t)$; the compactness follows from the first inequalities for $f_i$ and the fact that $f$ is a homeomorphism.
Since $f$ is injective, to prove \eqref{eq:dc2-7}, it suffices to show that $f(y)\in f(A)$.
Suppose otherwise.
Since $A$ is compact, one finds $z\in A$ such that $f(z)\neq f(y)$ is closest to $f(y)$.

Set $\tilde v:=f(y)-f(z)$ and consider a new strainer curve $\tilde c(s):=f^{-1}(f(z)+s\tilde v)$ passing through $z$.
Clearly, for sufficiently small $s>0$, we have
\begin{equation}\label{eq:dc2-8}
f_i(y)\le (f_i\circ\tilde c)(s)\le f_i(z)
\end{equation}
for all $1\le i\le n$.
In particular, $(f_i\circ\tilde c)^+(0)\le 0$.
Let $\tilde\gamma$ be a tangential shortest path for $\tilde c$ as in Lemma \ref{lem:cur3}.
By the second statement of Lemma \ref{lem:cur3}, we have
\[(f_i\circ\tilde\gamma)^+(0)\le 0.\]
By the assumption that $g$ is an $(\epsilon,\delta)$-special function, \eqref{eq:sp}, we get $(g\circ\tilde\gamma)^+(0)\ge\epsilon$.
By Lemma \ref{lem:cur3} again, this implies $(g\circ\tilde c)^+(0)>0$.
Thus, for sufficiently small $s>0$, we have
\begin{equation}\label{eq:dc2-9}
(g\circ\tilde c)(s)>g(z).
\end{equation}
The inequalities \eqref{eq:dc2-8} and \eqref{eq:dc2-9}, together with the fact that $z\in A$, imply $\tilde c(s)\in A$.
However, $(f\circ\tilde c)(s)$ is closer to $f(y)$ than $f(z)$, which is a contradiction.
\end{proof}

\begin{proof}[Proof of Proposition \ref{prop:dc}]
The proof is the same as in \cite[Section 3]{Per94}, \cite[Section 14.6]{LN19}.
Let $f:U\to\mathbb R^n$ be a strainer chart on $X$.
By the local semiconcavity, $f$ is a DC map.
It remains to show that $f^{-1}$ is a DC map.
We show that if $h$ is a semiconcave function on $U$, then $h\circ f^{-1}$ is a DC function.
For any fixed $x\in U$, let $g$ be an $(\epsilon_n,\delta_n)$-special function around $x$ as in Lemma \ref{lem:dc1}.
Consider the decomposition
\[h\circ f^{-1}=(h\circ f^{-1}+Lg\circ f^{-1})-Lg\circ f^{-1},\]
where $L\gg1$.
By Lemma \ref{lem:dc2}, the second term $Lg\circ f^{-1}$ is a semiconcave function.
Furthermore, if $L$ is large enough (compared to the local Lipschitz constant of $h$), then $(h+Lg)$ is also a $(1,\delta_n)$-special function for $f$.
Therefore, by Lemma \ref{lem:dc2}, the first term $(h+Lg)\circ f^{-1}$ is also semiconcave, as desired.
\end{proof}

We conclude this subsection with the following summary.

\begin{cor}\label{cor:dc}
Any locally semiconcave G-space admits canonical $C^{1,1/2}$ and DC structures defined by strainer charts.
Here, canonicity means that any squared distance function on each chart is a $C^{1,1/2}$ and DC function.
\end{cor}

\subsection{Finsler metric}\label{sec:fin}

Finally, we construct a Finsler metric for a locally semiconcave G-space in the same way as in \cite[Chapter 3]{Po98}, except for the $1/2$-H\"older continuity in Lemma \ref{lem:fin2}.

Let $X$ be a $C^1$ manifold and denote by $TX$ its tangent bundle.
By definition, a \emph{Finsler metric} is a function $F:TX\to\mathbb R$ such that its restriction to each fiber $T_xX$ is a norm for any $x\in X$.

We define a Finsler metric via the local strainer coordinates, pulling back the infinitesimal metric on $X$.

\begin{dfn}\label{dfn:fin}
Let $X$ be a locally semiconcave G-space and $f:U\to\mathbb R^n$ a strainer chart on $X$.
For any $x\in U$ and $v\in\mathbb R^n$ ($=T_{f(x)}\mathbb R^n$), we define
\[F(x,v):=|\dot c(0)|=\lim_{t\to0}\frac{|x,c(t)|}{|t|},\]
where $c(t)=f^{-1}(f(x)+tv)$ is a strainer curve as in Definition \ref{dfn:cur}.
\end{dfn}

See Lemma \ref{lem:cur2} for the existence of $|\dot c(0)|$.
To be more precise, $F(x,v)$ should be written as $F(f(x),v)$, but we use the above notation for simplicity.

The following is the main result of this subsection.

\begin{prop}\label{prop:fin}
The above function $F$ defines a continuous Finsler metric on $X$ that is compatible with the original distance.
Furthermore, $F$ is locally $1/2$-H\"older continuous in the base direction.
\end{prop}

The proof is carried out by checking the necessary conditions one by one in the following four lemmas.
We first check the following.

\begin{lem}\label{lem:fin1}
The function $F$ is well-defined and continuous on $TX$.
\end{lem}

\begin{proof}
Since the continuity of $|\dot c(0)|$ was already proved in Lemma \ref{lem:cur2}, it suffices to check the well-definedness of $F$.

Let $f:U\to\mathbb R^n$ be a strainer chart around $x\in X$ and $c(t)=f^{-1}(f(x)+tv)$ a strainer curve in the direction $v\in\mathbb R^n$.
Let $g:V\to\mathbb R^n$ be another strainer chart around $x$.
Set
\[\tilde c(t):=g^{-1}(g(x)+t\tilde v),\quad\tilde v:=d(g\circ f^{-1})(v)=(g\circ c)'(0).\]
We show that
\begin{equation}\label{eq:fin1}
|c(t),\tilde c(t)|<\delta(t)|t|,
\end{equation}
where $\delta$ is a positive function such that $\delta(t)\to0$ as $t\to0$.
Then, by the triangle inequality, we have $|\dot c(0)|=|\dot{\tilde c}(0)|$, as desired.
Since $g$ is bi-Lipschitz, \eqref{eq:fin1} is equivalent to $|g\circ c(t), g(x)+t\tilde v|<\delta(t)|t|$.
The latter immediately follows from the $C^1$ smoothness of $g\circ f^{-1}$ proved in Proposition \ref{prop:c1}.
\end{proof}

Using the explicit error function, one can improve the continuity of $F$ as follows.

\begin{lem}\label{lem:fin2}
$F$ is locally $1/2$-H\"older continuous in the base direction of $TX$: that is, the function $x\mapsto F(x,v)$ is locally $1/2$-H\"older continuous for any fixed $v\in\mathbb R^n$.
\end{lem}

\begin{proof}
The proof is similar to that of Proposition \ref{prop:c1}.
Let $f:U\to\mathbb R^n$ be a strainer chart on $X$.
Let $c_x(t)=f^{-1}(f(x)+tv)$ and $c_y(t)=f^{-1}(f(y)+tv)$ be strainer curves passing through $x,y\in U$ in the same direction $v\in\mathbb R^n$, respectively.
By Lemma \ref{lem:cur2}, we have
\begin{equation}\label{eq:fin2-1}
\left||\dot c_x(0)|-\frac{|x,c_x(t)|}{|t|}\right|<\delta(t)|v|,\quad\left||\dot c_y(0)|-\frac{|y,c_y(t)|}{|t|}\right|<\delta(t)|v|,
\end{equation}
By the $\epsilon_n$-bi-Lipschitz property of $f$ (Proposition \ref{prop:bilip}), we have
\begin{equation}\label{eq:fin2-2}
\left|\frac{|x,c_x(t)|}{|t|}-\frac{|y,c_y(t)|}{|t|}\right|\le\frac1{|t|}{(|xy|+|c_x(t),c_y(t)|)}\le\frac{2\epsilon_n^{-2}}{|t|}|xy|.
\end{equation}
Combining \eqref{eq:fin2-1} and \eqref{eq:fin2-2}, we obtain
\[\left||\dot c_x(0)|-|\dot c_y(0)|\right|<\frac{2\epsilon_n^{-2}}{|t|}|xy|+2\delta(t)|v|.\]
Since $\delta(t)=C|t||v|$ by Lemma \ref{lem:cur2}, letting $t:=|xy|^{1/2}$ shows the claim.
\end{proof}

Next we show the following.

\begin{lem}\label{lem:fin3}
The function $F$ is a (reversible) norm on $T_xX$ for any $x\in X$.
\end{lem}

\begin{proof}
In what follows, we fix $x$ and consider the restriction of $F$ to $T_xX$.
Since a strainer chart is bi-Lipschitz, $F$ is positive definite.
By definition, $F$ is absolutely homogeneous.
It remains to prove the triangle inequality.
Let $v,w\in T_xX$.
For small $t>0$, set
\begin{align*}
c(t)&:=f^{-1}(f(x)+tv),\\
\tilde c(t)&:=f^{-1}(f(c(t))+tw)=f^{-1}(f(x)+t(v+w)).
\end{align*}
See Figure \ref{fig:fin3}.

\begin{figure}[ht]
\centering
\begin{tikzpicture}
\coordinate[label=below:$x$](x)at(0,0);
\coordinate[label=below:$c(t)$](ct)at(2,0);
\coordinate[label=above:$\tilde c(t)$](ct')at(3,2);

\coordinate[label=below:$v$](v)at(1,0);
\coordinate[label=below right:$w$](w)at(2.5,1);
\coordinate[label=above left:$v+w$](v+w)at(1.5,1);

\draw[dashed](x)to(ct);
\draw[dashed](ct)to(ct');
\draw[dashed](x)to(ct');

\draw[->, >=Stealth](x)to(v);
\draw[->, >=Stealth](ct)to(w);
\draw[->, >=Stealth](x)to(v+w);

\fill(x)circle(1.5pt);
\fill(ct)circle(1.5pt);
\fill(ct')circle(1.5pt);
\end{tikzpicture}
\caption{}\label{fig:fin3}
\end{figure}

By the triangle inequality for the original distance, we have
\[|x,\tilde c(t)|\le|x,c(t)|+|c(t),\tilde c(t)|.\]
Dividing both sides by $t$ and taking $t\to0$, we get
\begin{equation}\label{eq:fin3-1}
F(x,v+w)\le F(x,v)+\lim_{t\to0}\frac{|c(t),\tilde c(t)|}{t}.
\end{equation}
Furthermore, by the uniform convergence of norm in Lemma \ref{lem:cur2}, we see that
\[\left|\frac{|c(t),\tilde c(t)|}{t}-F(c(t),w)\right|<\delta(t)|w|.\]
Taking $t\to0$ and using the continuity of $F$, we obtain
\begin{equation}\label{eq:fin3-2}
\lim_{t\to0}\frac{|c(t),\tilde c(t)|}{t}=F(x,w).
\end{equation}
Combining \eqref{eq:fin3-1} and \eqref{eq:fin3-2} completes the proof.
\end{proof}

Finally, we prove the compatibility of the Finsler metric with the original distance.

\begin{lem}\label{lem:fin4}
The Finsler metric $F$ induces the original distance $d$ of $X$:
that is, if we define a new distance $\tilde d$ of $X$ by
\[\tilde d(p,q):=\inf_{c}\tilde L(c),\quad \tilde L(c):=\int_0^1F(c(t),c'(t))dt,\]
where $c:[0,1]\to X$ runs over all $C^1$ curves between $p$ and $q$, then we have $d=\tilde d$.
\end{lem}

\begin{proof}
Since $d$ is an intrinsic distance, it is written as
\[d(p,q)=\inf_{c}L(c),\quad L(c):=\sup\sum_{i=1}^Nd(c(t_{i-1}),c(t_i)),\]
where $c:[0,1]\to X$ runs over all (not necessarily $C^1$) curves between $p$ and $q$, and the supremum is taken over all finite subdivisions $0=t_0<t_1<\dots<t_N=1$.
Since any shortest path in $X$ is a $C^1$ curve by Proposition \ref{prop:c1}, in order to prove that $d=\tilde d$, it suffices to show that $L(c)=\tilde L(c)$ for any $C^1$ curve $c$.

Let $c:[0,1]\to X$ be a $C^1$ curve.
Take a sufficiently fine subdivision $0=t_0<t_1<\dots<t_N=1$ and set $\Delta:=\max_{1\le i\le N}(t_i-t_{i-1})\ll1$.
By the uniform convergence of norm in Lemma \ref{lem:cur2} and the $C^1$ smoothness of $c$, we have
\[\left||c(t_{i-1}),c(t_i)|-F(c(t_{i-1}),c'(t_{i-1}))(t_i-t_{i-1})\right|<M\delta(\Delta)(t_i-t_{i-1}),\]
where $M$ is a positive number depending on the maximum speed of $c$.
By adding all the inequalities and taking $\Delta\to0$, we obtain $L(c)=\tilde L(c)$.
\end{proof}

Proposition \ref{prop:fin} follows from Lemmas \ref{lem:fin1}, \ref{lem:fin2}, \ref{lem:fin3}, and \ref{lem:fin4}.
Now Theorem \ref{thm:conc} follows from Corollaries \ref{cor:str} and \ref{cor:dc}, and Proposition \ref{prop:fin}.

\begin{rem}
Our result is still a little weaker than the original work of Otsu--Shioya \cite{OS94}, even taking into account the absence of the Riemannian structure.
In \cite[Lemma 3.2(2)]{OS94}, they proved that the cosine of the angle function in an Alexandrov G-space is locally $1/2$-H\"older continuous.
To be precise, we have not shown this statement, and the original proof seems to rely on the Riemannian property of Alexandrov spaces.
\end{rem}

\section{Semiconvex case}\label{sec:conv}

In this section, we prove Theorem \ref{thm:conv} for locally semiconvex, locally differentiable G-spaces.
The proof is almost the same as the semiconcave case, just by reversing the previous inequalities in an appropriate way, except for the proof of the uniform convergence of angle; see Proposition \ref{prop:unif'}.

We begin with the following basic remark.
As already mentioned in Remark \ref{rem:main}, local differentiability does not follow from local semiconvexity.
However, the following holds.

\begin{lem}\label{lem:cat}
Any locally CAT G-space $X$ is locally differentiable.
\end{lem}

\begin{proof}
This follows from the fact that every space of directions of $X$ is a unit sphere (see \cite{Be02,LN19}).
Indeed, let $p,x,y,z\in X$ be points in a sufficiently small ball such that $y,x,z$ lie on a single shortest path.
Suppose that the local differentiability does not hold for these points.
Together with Lemma \ref{lem:sum'}, this implies
\[\angle pxy+\angle pxz>\pi.\]
Consider the space of directions $\Sigma_xX$ at $x$ and denote by $p',y',z'\in\Sigma_xX$ the elements defined by the shortest paths $xp, xy, xz$, respectively.
Since our angle is compatible with the distance on the space of directions in the CAT setting, the above inequality implies $|p',y'|+|p',z'|>\pi$.
Since $y'$ and $z'$ are antipodal points in $\Sigma_xX$, this contradicts the fact that $\Sigma_xX$ is a unit sphere.
\end{proof}

Now, we will verify step by step that all the results of Section \ref{sec:key} are valid in the current setting of locally semiconvex, locally differentiable G-spaces.
\begin{itemize}
\item Lemma \ref{lem:diff} is nothing but local differentiability, our current assumption.
\item Remark \ref{rem:diff1} is proved by the same argument as in the proof of Lemma \ref{lem:diff}, using Lemma \ref{lem:gen1} and the local differentiability assumption.
\item Lemmas \ref{lem:conti} and \ref{lem:well} are proved in exactly the same way by using the local differentiability and Lemmas \ref{lem:semiconti'} and \ref{lem:gen2}, respectively.
\end{itemize}
The details are left to the reader.

Moreover, we show that any locally semiconvex, locally differentiable G-space satisfies the uniform convergence of angle similar to Proposition \ref{prop:unif}.
However, the statement, as well as the proof, is somewhat different.
Compare also with \cite[Proposition 3.5]{Be02} and \cite[Lemma 7.6]{LN19} in the CAT setting, whose proofs rely on the Riemannian property of CAT spaces.

\begin{prop}\label{prop:unif'}
Let $X$ be a locally semiconvex, locally differentiable G-space and $B$ a normal ball.
Let $p\in B$ and $K\subset B\setminus\{p\}$ a compact set.
Then for any $x,y\in K$, we have
\[|\angle pxy-\tilde \angle pxy|<\delta(|xy|),\]
where $\delta(s)$ is a positive function independent of $x$ and $y$ (but depending on $p$ and $K$) such that $\delta(s)\to 0$ as $s\to 0$.
\end{prop}

Note that, unlike Proposition \ref{prop:unif}, the error function $\delta$ has no explicit form and is highly dependent on the local geometry of $X$.
This is why we cannot prove the $1/2$-H\"older continuity in Theorem \ref{thm:conv}, unlike in Theorem \ref{thm:conc}.

\begin{proof}
Choose $l>0$ such that the closed $l$-neighborhood of $K$ is contained in $B$.
Extend the shortest path $xy$ beyond $y$ and let $y'\in B$ be a new endpoint such that $|xy'|=l$.
See Figure \ref{fig:unif'}.

\begin{figure}[ht]
\centering
\begin{tikzpicture}
\coordinate[label=below:$p$](p)at(-0.5,-4);
\coordinate[label=above left:$x$](x)at(0,0);
\coordinate[label=above:$y$](y)at(1,0);
\coordinate[label=above right:$y'$](y')at(3,0);

\draw(p)to(x);
\draw(x)to(y)to(y');
\draw(p)to(y);
\draw[dashed](p)to(y');

\draw(0.5,0)circle[radius=1.5];
\node at(0.5,1){$K$};

\fill(p)circle(1.5pt);
\fill(x)circle(1.5pt);
\fill(y)circle(1.5pt);
\fill(y')circle(1.5pt);

\pic[draw, angle radius=3mm] {angle = p--x--y};
\pic[draw, angle radius=3mm] {angle = x--y--p};
\pic[draw, angle radius=3mm] {angle = p--y--y'};
\end{tikzpicture}
\caption{}\label{fig:unif'}
\end{figure}

By Lemma \ref{lem:sum'}, we have
\begin{equation}\label{eq:unif'1}
\angle pyx+\angle pyy'\ge \pi.
\end{equation}
Due to local differentiability, the angle varies continuously (cf.\ Lemma \ref{lem:conti}).
Therefore, by a compactness argument using the Arzel\`a--Ascoli theorem \cite[Theorem 2.5.14]{BBI01}, we get the following: for any positive number $\delta$, if the distance between $x$ and $y$ is small enough, then
\begin{equation}\label{eq:unif'2}
|\angle pxy-\angle pyy'|<\delta.
\end{equation}
Indeed, since $\angle pxy=\angle pxy'$, the above two angles converge to the same angle when $x$ and $y$ converge to the same point $z\in K$ and $xy'$ converges to some shortest path emanating from $z$.
Hence \eqref{eq:unif'2} is proved by contradiction.
Note that here we used the assumptions that $x,y\in K$ and $|xy'|=l$ to ensure that the limit triangle does not degenerate at the limit point $z$.
Furthermore, by the almost comparison inequality \eqref{eq:comp'}, we have
\begin{equation}\label{eq:unif'3}
\angle pxy+\angle pyx\le\tilde\angle pxy+\tilde\angle pyx+\delta\le\pi+\delta,
\end{equation}
provided that $|xy|/|pK|$ is small enough compared to $\delta$.
Combining the inequalities \eqref{eq:unif'1}, \eqref{eq:unif'2}, and \eqref{eq:unif'3}, we obtain the desired inequality.
\end{proof}

\begin{rem}\label{rem:unif'}
One might wonder if it is possible to remove the dependence of $\delta$ on $p$ and $K$ and improve the error function $\delta(|xy|)$ to the form $\delta(|xy|/|px|)$ as in Proposition \ref{prop:unif}, even if it is an implicit one.
The reason we cannot do this is as follows.
If $p$ is not bounded away from $x,y$ by a fixed compact set $K$, then when $x,y$ and $p$ converge to the same point $z$, it is unclear whether the extensions of the shortest paths $xp$ and $yp$ beyond $p$ converge to the same shortest path emanating from $z$, even if $|xy|/|px|$ tends to $0$.
\end{rem}

For comparison with the GCBA case in \cite[Lemma 7.6]{LN19} (not necessarily a G-space), we provide a slightly different proof of Proposition \ref{prop:unif'}.
This has a somewhat similar flavor to the proof of Proposition \ref{prop:unif} in the semiconcave case.

\begin{proof}[Alternative proof of Proposition \ref{prop:unif'}]
As before, choose $l>0$ such that the closed $l$-neighborhood of $K$ is contained in $B$.
Extend the shortest path $px$ beyond $x$ and let $q\in B$ be a new endpoint such that $|qx|=l$.
See Figure \ref{fig:unif''}.

\begin{figure}[ht]
\centering
\begin{tikzpicture}
\coordinate[label=below:$p$](p)at(-0.5,-3);
\coordinate[label=above left:$x$](x)at(0,0);
\coordinate[label=above right:$y$](y)at(1,0);
\coordinate[label=above:$q$](q)at(0.5,3);
\coordinate[label=above left:$x'$](x')at(-2,0);

\draw(p)to(x)to(q);
\draw(x')to(x)to(y);
\draw(p)to(y);
\draw(q)to(y);

\draw(0.5,0)circle[radius=1.5];
\node at(1.5,-0.5){$K$};

\fill(p)circle(1.5pt);
\fill(x)circle(1.5pt);
\fill(y)circle(1.5pt);
\fill(q)circle(1.5pt);
\fill(x')circle(1.5pt);

\pic[draw, angle radius=3mm] {angle = p--x--y};
\pic[draw, angle radius=3mm] {angle = x--y--p};
\pic[draw, angle radius=3mm] {angle = y--x--q};
\pic[draw, angle radius=3mm] {angle = q--y--x};
\end{tikzpicture}
\caption{}\label{fig:unif''}
\end{figure}

By Lemma \ref{lem:gen1}, we have
\begin{equation}\label{eq:unif''1}
\angle pxy+\angle qxy\ge\pi
\end{equation}
(in fact the equality holds, due to local differentiability and Remark \ref{rem:diff1}).
On the other hand, for any positive number $\delta$, if $x$ and $y$ are sufficiently close, then
\begin{equation}\label{eq:unif''2}
\angle pyx+\angle qyx\ge\pi-\delta.
\end{equation}
This immediately follows from \eqref{eq:unif'2} and Lemma \ref{lem:gen1}.
Indeed, take a point $x'\in B$ on the extension of the shortest path $yx$ beyond $x$ with $|x'y|=l$.
Then by the same argument as \eqref{eq:unif'2}, the two angles in \eqref{eq:unif''2} are $\delta/2$-close to $\angle pxx'$ and $\angle qxx'$, respectively, provided $|xy|$ is small enough.
The sum of the latter two angles is not less than $\pi$ by Lemma \ref{lem:gen1}, which gives \eqref{eq:unif''2}.
Furthermore, by the almost comparison inequality \eqref{eq:comp'} as before, we have
\begin{gather}
\angle pxy+\angle pyx\le\tilde\angle pxy+\tilde\angle pyx+\delta\le\pi+\delta,\label{eq:unif''3}\\
\angle qxy+\angle qyx\le\tilde\angle qxy+\tilde\angle qyx+\delta\le\pi+\delta,\label{eq:unif''4}
\end{gather}
provided $|xy|$ is small enough compared to $|pK|$ and $l$.
Combining \eqref{eq:unif''1}, \eqref{eq:unif''2}, \eqref{eq:unif''3}, and \eqref{eq:unif''4}, we obtain the desired inequality.
Note that the difference from the previous proof is that we used Lemma \ref{lem:gen1} here instead of Lemma \ref{lem:sum'}.
\end{proof}

Now the rest of the proof of Theorem \ref{thm:conv} proceeds almost the same way as that of Theorem \ref{thm:conc}, except for the part on the $1/2$-H\"older continuity.
We will briefly explain the necessary modifications.

\begin{proof}[Proof of Theorem \ref{thm:conv}]
In what follows, we will consider the statements in Sections \ref{sec:str}, \ref{sec:cur}, \ref{sec:dc}, and \ref{sec:fin} with local semiconcavity replaced by local semiconvexity plus local differentiability, as already done for Section \ref{sec:key}.

Regarding Section \ref{sec:str}, the definition and properties of a strainer are exactly the same as before, since it only involves the first-order derivative of the distance function, i.e., angles.
Note that the strainer argument takes place in a small neighborhood of a fixed strained point (compared to the length of a strainer), so the additional dependency of the error function $\delta$ of Proposition \ref{prop:unif'} does not matter.

Regarding Section \ref{sec:cur}, the properties of a strainer curve remain true with the same proofs, except for the explicit representation of the error function $\delta$.
Every statement holds for an implicit error function arising from Proposition \ref{prop:unif'}.

Regarding Section \ref{sec:dc}, due to the lack of the explicit error function, we only get a $C^1$ structure in Proposition \ref{prop:c1}, which immediately follows from Lemma \ref{lem:cur4}.
On the other hand, Proposition \ref{prop:dc} on the DC structure remains true, but requires slight modifications, as it involves the second-order derivative of the distance function.
In fact, it suffices to change the sign in the following manner: a general semiconvex function $f$ is defined by the condition that $-f$ is semiconcave; a DC function here should be defined as the difference of two semiconvex functions, which in turn is equivalent to the previous definition; then all the previous arguments carry over to the current setting simply by replacing a function $f$ with $-f$.
For example, the definition of an $(\epsilon,\delta)$-special function in \eqref{eq:sp} will be replaced by
\[(f_i\circ\gamma)^+(0)\ge-\delta\text{ for all }1\le i\le n\implies (g\circ\gamma)^+(0)\le-\epsilon.\]
Then Lemmas \ref{lem:dc1} and \ref{lem:dc2} are valid for locally semiconvex, locally differentiable G-spaces, by the same proofs with the sign change explained above.
The details are left to the reader.

Regarding Section \ref{sec:fin}, we can repeat exactly the same construction of a Finsler metric, except for Lemma \ref{lem:fin2} on the $1/2$-H\"older continuity.
This completes the proof of Theorem \ref{thm:conv}.
\end{proof}

Finally, we give further remarks on the work of Pogorelov \cite[Chapter 3]{Po98}.
Compare with Remark \ref{rem:pog1}.

\begin{rem}\label{rem:pog2}
Similar to Remark \ref{rem:pog1}, the above argument seems to show that any locally semiconvex, locally differentiable G-space satisfies Axiom A in \cite[Chapter 3]{Po98}.
This is nearly correct in the sense that we have proved the existence of a well-defined angle and its continuous dependence on shortest paths (cf.\ Lemmas \ref{lem:conti} and \ref{lem:well}), as well as the sum formulae for adjacent/opposite angles (cf.\ Definition \ref{dfn:diff} and Remark \ref{rem:diff1}).
However, strictly speaking, due to the weak error function in Proposition \ref{prop:unif'} (cf.\ Remark \ref{rem:unif'}), it is still unclear whether the ``slope'' in the sense of \cite{Po98} is well-defined (because the slope allows the degeneration of a comparison triangle, unlike our definition of angle).
However, as we saw above, this does not affect the construction of a Finsler metric.
\end{rem}

\begin{rem}\label{rem:pog3}
Conversely, an argument similar to Proposition \ref{prop:unif'} shows that Axiom A implies the uniform convergence of angle, which seems to be implicitly used in \cite{Po98}.
Here we rely on the mean value theorem instead of the local semiconvexity used above.

Let $X$ be a G-space satisfying Axiom A, $B$ a normal ball, $p\in B$, and $K\subset B\setminus\{p\}$ compact.
For any $x,y\in K$, if $|xy|/|pK|$ is small enough, then $\cos\tilde\angle pxy$ is close to the differential quotient $(|px|-|py|)/|xy|$.
By the mean value theorem, this differential quotient is equal to some $\cos\angle pzy$, where $z$ is a point on the shortest path $xy$.
Furthermore, if $x$ and $y$ are sufficiently close, then the same argument as \eqref{eq:unif'2} using Axiom A shows that $\angle pzy$ is close to $\angle pxy$.
Consequently, $\tilde\angle pxy$ is uniformly close to $\angle pxy$, independently of $x$ and $y$.
\end{rem}

\section{Non-differentiable case}\label{sec:top}

In this section, we prove Theorem \ref{thm:top} for (not necessarily locally differentiable) locally semiconvex G-spaces.

Recall that every G-space is topologically homogeneous; see Section \ref{sec:g} (the proof can be found in \cite[Theorem 2.5]{Th96}, \cite[Corollary 3.12]{BHR11}).
Therefore, to prove that a locally semiconvex G-space is a topological manifold, it is sufficient to find at least one manifold point: it is not necessary to construct strainer coordinates everywhere as before.

To find a manifold point, we generalize the strainer argument developed in Section \ref{sec:str}.
Note that, compared to the previous settings with local differentiability discussed in Sections \ref{sec:conc} and \ref{sec:conv}, we do not have the continuity of angle (Lemma \ref{lem:conti}), well-definedness of angle (Lemma \ref{lem:well}), or the uniform convergence of angle (Propositions \ref{prop:unif} and \ref{prop:unif'}).
We can only use the general properties of angle established in Section \ref{sec:ang}.

Nevertheless, it is still possible to carry out a weaker version of a strainer argument.
In fact, such a strainer argument (which does not rely on differentiability) has already been developed in the recent study of Busemann convex spaces by the authors \cite{FG25}.
The following argument is based on the same idea as in \cite{FG25}.

We first modify the definition of a strainer to deal with the non-differentiability of a locally semiconvex G-space.
Compare with Definition \ref{dfn:str}.

\begin{dfn}\label{dfn:str'}
Let $X$ be a locally semiconvex G-space, $B$ a normal ball, and $x\in B$.
A collection of points $p_1,\dots,p_k\in B\setminus\{x\}$ is called a \emph{$(k,\delta)$-strainer} at $x$ if it satisfies
\begin{equation}\label{eq:str'}
|\angle p_ixp_j-\pi/2|<\delta,\quad|\angle p_ixq_j-\pi/2|<\delta
\end{equation}
for any $1\le i< j\le k$, where $q_j\in B$ is a point on the extension of the shortest path $p_jx$ beyond $x$.
The point $x$ is called a \emph{$(k,\delta)$-strained point} and the distance map $f:=(d(p_1,\cdot),\dots,d(p_k,\cdot))$ is called a \emph{$(k,\delta)$-strainer} map at $x$.
\end{dfn}

Here we assume that $\delta$ is sufficiently small.
As before, an upper bound for $\delta$ will be defined later in Proposition \ref{prop:open'} in terms of $k$.
Note that the angle $\angle p_ixq_j$ does not depend on a choice of $q_j$, though it depends on $x$ (this treatment is slightly different from \cite{FG25}).

\begin{rem}
If, in addition, $X$ is locally differentiable, then the second inequality in \eqref{eq:str'} for $\angle p_ixq_j$ follows from the first one for $\angle p_ixp_j$, since their sum is equal to $\pi$ (see Definition \ref{dfn:diff'}).
Therefore Definition \ref{dfn:str'} generalizes Definition \ref{dfn:str} in the locally semiconvex, locally differentiable case (which was omitted in Section \ref{sec:conv}.)
In particular, the condition \eqref{eq:str'} can be viewed as the condition \eqref{eq:str} plus ``almost differentiability'' of the distance function from $p_i$ along the shortest path $p_jq_j$ in the sense that $|\angle p_ixp_j+\angle p_ixq_j-\pi|<2\delta$.
\end{rem}

Although we cannot rely on the continuity of angle, we still have the following openness of the strainer condition.
Compare with Remark \ref{rem:stropen}.

\begin{lem}\label{lem:stropen'}
If $p_1,\dots,p_k$ is a $(k,\delta)$-strainer at $x$ as in Definition \ref{dfn:str'}, then it is a $(k,\delta)$-strainer at any point in a neighborhood of $x$.
\end{lem}

\begin{proof}
Assume that $y\in B$ is sufficiently close to $x$.
By the upper semicontinuity of angle, Lemma \ref{lem:semiconti'}, we have
\[\angle p_iyp_j<\pi/2+\delta,\quad \angle p_iyq_j'<\pi/2+\delta,\]
where $q_j'\in B$ is a point on the extension of the shortest path $p_jy$ beyond $y$.
On the other hand, by Lemma \ref{lem:sum'}, the sum of these two angles is not less than $\pi$.
Therefore, we also have $\angle p_iyp_j>\pi/2-\delta$ and $\angle p_iyq_j'>\pi/2-\delta$, as required.
\end{proof}

\begin{rem}
The above proof shows that in the current setting of local semiconvexity, only upper bounds for angles are necessary in the almost orthogonality condition \eqref{eq:str'}.
In fact it is possible to define a strainer in that way (similar to \cite{LN19}), but we did not take this approach for comparison with the previous sections.
\end{rem}

The following proposition is proved in exactly the same way as Proposition \ref{prop:open}, using the almost orthogonality condition \eqref{eq:str'}.
See also Remark \ref{rem:open} and compare with \cite[Proposition 5.17]{FG25}.

\begin{prop}\label{prop:open'}
Let $X$ be a locally semiconvex G-space and $B$ a normal ball.
Let $f$ be a $(k,\delta)$-strainer map at $x\in B$.
If $\delta\le\delta_k$, then $f$ is an $\epsilon_k$-open map in a neighborhood of $x$, where $\delta_k\ll\epsilon_k$ are positive constants depending only on $k$.
\end{prop}

The next lemma is the key to dealing with the non-differentiability of a locally semiconvex G-space.
Compare with Lemma \ref{lem:bilip} and \cite[Proposition 5.7]{FG25}.

\begin{lem}\label{lem:bilip'}
Let $X$ be a locally semiconvex G-space and $B$ a normal ball.
Let $p_1,\dots,p_k\in B$ be a $(k,\delta)$-strainer at $x\in B$.
Suppose that there exists $y\in B$ sufficiently close to $x$ such that
\[||p_ix|-|p_iy||<\delta'|xy|\]
for all $1\le i\le k$, where $\delta'\ll\delta$ is a positive constant depending only on $\delta$.
Then $p_1,\dots,p_k,x$ is a $(k+1,\delta)$-strainer at $y$.
\end{lem}

Note that, unlike Lemma \ref{lem:bilip} where $y$ and $z$ are arbitrary points, here we have to fix $z=x$.
Otherwise, the statement does not hold: consider the case where $x$ is a non-differentiability point of a strictly convex norm on $\mathbb R^n$.
It should be emphasized that the proof of Lemma \ref{lem:bilip} in the locally semiconvex, locally differentiable case (which was omitted in Section \ref{sec:conv}) relied on the uniform convergence of angle, Proposition \ref{prop:unif'}.

\begin{proof}[Proof of Lemma \ref{lem:bilip'}]
By the openness of the strainer condition, Lemma \ref{lem:stropen'}, we may assume that $p_1,\dots,p_k$ is a $(k,\delta)$-strainer at $y$.
Hence it remains to show that
\begin{equation}\label{eq:bilip'1}
|\angle p_iyx-\pi/2|<\delta,\quad|\angle p_iyy'-\pi/2|<\delta,
\end{equation}
where $y'$ is a point on the extension of the shortest path $xy$ beyond $y$.

In what follows, we repeat the same argument as in the proof of Proposition \ref{prop:unif'}, replacing the continuity of angle (cf.\ Lemma \ref{lem:conti}) with the upper semicontinuity of angle, Lemma \ref{lem:semiconti'}.
See again Figure \ref{fig:unif'}, regarding $p$ as $p_i$ and ignoring $K$.

By Lemma \ref{lem:sum'} on adjacent angles, we have
\[\angle p_iyx+\angle p_iyy'\ge \pi.\]
Since $y$ is sufficiently close to $x$, the upper semicontinuity of angle in Lemma \ref{lem:semiconti'} implies
\[\angle p_iyy'\le \angle p_ixy+\delta',\]
where $\delta'$ is as in the statement of Lemma \ref{lem:bilip'} (compare the above inequality with \eqref{eq:unif'2} where $x$ is not fixed).
By the almost comparison inequality \eqref{eq:comp'}, we have
\[\angle p_ixy+\angle p_iyx\le\tilde\angle p_ixy+\tilde\angle p_iyx+\delta'\le\pi+\delta',\]
provided that $|xy|/|p_ix|$ is small enough compared to $\delta'$.
Combining the above three inequalities, we obtain
\begin{equation}\label{eq:bilip'2}
|\angle p_iyx-\tilde\angle p_iyx|<2\delta',\quad|\angle p_i yx+\angle p_iyy'-\pi|<2\delta'.
\end{equation}
Now, using the assumption that $||p_ix|-|p_iy||<\delta'|xy|$, we have $|\tilde\angle p_iyx-\pi/2|<2\delta'$, provided $\delta'$ is small enough.
Together with \eqref{eq:bilip'2}, this implies the desired inequality \eqref{eq:bilip'1}, provided $\delta'\ll\delta$.
\end{proof}

For the proof of Theorem \ref{thm:top}, we need the following lemma.
Compare with \cite[Proposition 5.3, Corollary 5.4]{LN19} and \cite{Be77, BHR11} (see also Remarks \ref{rem:dim1} and \ref{rem:dim2} below).

\begin{lem}\label{lem:dim}
Any normal ball in a locally semiconvex G-space admits a bi-Lipschitz embedding into Euclidean space (which is an almost isometry with respect to the sup-norm).
In particular, any normal ball has finite Hausdorff dimension.
\end{lem}

\begin{proof}
The proof is along the same lines as that of \cite[Proposition 5.3]{LN19}.
Let $X$ be a locally semiconvex G-space and $B$ a normal ball of radius $r>0$. 
Let us denote by $10\bar B$ the concentric closed ball of radius $10r$.

Let $\delta>0$ be small enough, which will be determined later.
Take a finite $\delta$-net $\{p_\alpha\}_{\alpha=1}^N$ of $10\bar B$, i.e., the closed $\delta$-balls around $p_\alpha$ cover $10\bar B$.
We prove that the distance map
\[f:=(d(p_1,\cdot),\dots,d(p_N,\cdot)):B\to\mathbb R^N\]
is a locally $(1-\varkappa(\delta))$-bi-Lipschitz embedding with respect to the sup-norm of $\mathbb R^N$, where $\varkappa(\delta)$ is a positive function such that $\varkappa(\delta)\to0$ as $\delta\to 0$.
Note that $f$ is clearly $1$-Lipschitz.

In what follows, we abuse $\varkappa(\delta)$ to denote various functions satisfying the above property.
Let $x,y\in B$ be such that $|xy|\le\delta$.
Extend the shortest path $yx$ beyond $x$ and let $z,w\in 10\bar B$ be points on the extension such that $|xz|=r'$, $|xw|=9r$, respectively, where $\delta\ll r'\ll r$ will be determined later.
Since $\{p_\alpha\}_{\alpha=1}^N$ is a $\delta$-net, there exists $1\le \alpha\le N$ such that $p_\alpha$ is $\delta$-close to $w$.
See Figure \ref{fig:dim}.

\begin{figure}[ht]
\centering
\begin{tikzpicture}
\coordinate[label=above:$x$](x)at(0,0);
\coordinate[label=above:$y$](y)at(1,0);
\coordinate[label=above:$z$](z)at(-2,0);
\coordinate[label=above:$w$](w)at(-6,0);
\coordinate[label=below:$p_\alpha$](p)at(-6,-1);

\draw(w)to(z)to node[above]{$r'$}(x)to node[above]{$\delta$}(y);
\draw(p)to(x);
\draw[dashed](p)to node[left]{$\delta$}(w);
\draw[dotted](-5.9,0.5) .. controls (-4,1) and (-2,1) .. node[fill=white]{$9r$}(-0.1,0.5);

\fill(p)circle(1.5pt);
\fill(x)circle(1.5pt);
\fill(y)circle(1.5pt);
\fill(z)circle(1.5pt);
\fill(w)circle(1.5pt);

\pic[draw, angle radius=3mm] {angle = w--x--p};
\pic[draw, angle radius=3mm] {angle = p--x--y};
\end{tikzpicture}
\caption{}\label{fig:dim}
\end{figure}

Since $\tilde\angle wxz=0$, we have
\[\tilde\angle p_\alpha xz<\varkappa(\delta),\]
provided that $\delta$ is sufficiently small compared to $r'$.
On the other hand, the almost comparison inequality \eqref{eq:comp'} and Lemma \ref{lem:sum'} on adjacent angles imply
\[\tilde\angle p_\alpha xz+\tilde\angle p_\alpha xy\ge\angle p_\alpha xz+\angle p_\alpha xy-\varkappa(\delta)\ge\pi-\varkappa(\delta),\]
provided that $r'/r$ is sufficiently small compared to the local semiconvexity constant on $10\bar B$.
Combining the above two inequalities, we obtain
\[\tilde\angle p_\alpha xy\ge\pi-\varkappa(\delta).\]
Since $|xy|\le\delta\ll r$, this implies
\[|p_\alpha y|-|p_\alpha x|\ge (1-\varkappa(\delta))|xy|.\]
Therefore, $f$ is locally $(1-\varkappa(\delta))$-bi-Lipschitz with respect to the sup-norm.

Finally, we modify $f$ to a global $(1-\varkappa(\delta))$-bi-Lipschitz embedding.
Take a finite $\delta'$-net $\{q_\beta\}_{\beta=1}^M$ in $\bar B$, where $0<\delta'\ll\delta$, and consider the distance map
\[g:=(d(q_1,\cdot),\dots,d(q_M,\cdot)):B\to\mathbb R^M.\]
Then the triangle inequality shows that $g$ satisfies the $(1-\varkappa(\delta'/\delta))$-bi-Lipschitz property for any $x,y\in B$ with $|xy|\ge\delta$.
Therefore the map $(f,g):B\to\mathbb R^{N+M}$ gives a global almost isometric embedding.
This completes the proof.
\end{proof}

\begin{rem}\label{rem:dim1}
The above proof does not rely on the non-branching property of a G-space, or in other words, the uniqueness of the extension of a shortest path.
In particular, it applies to GNPC spaces discussed in \cite{FG25}: any tiny ball in a GNPC space admits an almost isometric embedding into Euclidean space with sup-norm (here a ``tiny ball'' plays the same role as a normal ball in this paper).
This is a natural extension of \cite[Proposition 5.3]{LN19}.
Compare also with \cite[Proposition 3.1, Remark 3.4]{FG25}.
\end{rem}

\begin{rem}\label{rem:dim2}
Instead of the local semiconvexity, if we assume a qualitative condition that every small metric ball is convex, then the same construction as above gives a \emph{topological} embedding into Euclidean space; see \cite{Be77, BHR11}.
\end{rem}

Now we are in a position to prove Theorem \ref{thm:top}.

\begin{proof}[Proof of Theorem \ref{thm:top}]
Let $X$ be a locally semiconvex G-space and $B$ a normal ball.
By the topological homogeneity of $X$, it suffices to find a manifold point in $B$.

By Lemma \ref{lem:dim}, $B$ has finite Hausdorff dimension, say at most $N$.
Let
\[\delta:=\min_{1\le k\le N+1}\delta_k,\]
where $\delta_k$ is a constant of Proposition \ref{prop:open'} depending only on $k$.
By Proposition \ref{prop:open'}, if $B$ contains a $(k,\delta)$-strained point, where $1\le k\le N+1$, then the Hausdorff dimension of $B$ is at least $k$.
Hence there exists a maximal strained point $x$ in $B$, that is, $x$ is an $(n,\delta)$-strained point and there are no $(n+1,\delta)$-strained points in $B$, where $1\le n\le N$.

Let $f:U\to\mathbb R^n$ be an $(n,\delta)$-strainer map around $x$.
By Lemma \ref{lem:bilip'} and the above choice of $n$,  we have for any $y\in U$,
\begin{equation}\label{eq:liminf}
\liminf_{z\to y}\frac{|f(y)-f(z)|}{|yz|}>\delta',
\end{equation}
where $\delta'$ is a positive constant depending only on $\delta$.

Now, applying \cite[Proposition 1.1]{Ly05}, we see that there exists an open (dense) subset of $U$ on which $f$ is a locally bi-Lipschitz embedding.
Since $f$ is open by Proposition \ref{prop:open'}, this gives a manifold point.
For the convenience of the reader, we recall an outline of \cite[Proposition 1.1]{Ly05} below.

For any positive integer $m$, we define a closed subset $U_m$ of $U$ by
\[U_m:=\{y\in U\mid |yz|<1/m,\ z\in U \implies|f(y)-f(z)|\ge\delta'|yz|.\}\]
By the inequality \eqref{eq:liminf}, $\{U_m\}_m$ covers $U$.
Clearly, $f$ is locally bi-Lipschitz on each $U_m$.
Consider a countable covering of $U$ by locally closed subsets $V_m:=U_m\setminus U_{m-1}$.
Then the Baire category theorem implies that the union of the interior of $V_m$ is dense in $U$ (see \cite[Lemma 2.1]{Ly05}).
This completes the proof.
\end{proof}

\begin{rem}\label{rem:gnpc}
In \cite[Theorem 5.30]{FG25}, the authors proved that any maximal strainer map on a GNPC space is an (open) bi-Lipschitz embedding, despite the lack of the uniform convergence of angle, using the local contractibility of fibers.
See also \cite[Remark 5.32]{FG25}.
\end{rem}

\section{Problems}\label{sec:prob}

We conclude this paper by proposing some open problems.

The first question concerns the BV property of our Finsler metric.
Recall that in the case of Alexandrov or locally CAT G-spaces, the components of the Riemannian metric have locally bounded variation; see \cite{Per94, LN19}.

\begin{ques}
Do the Finsler metrics of Theorems \ref{thm:conc} and \ref{thm:conv} have any bounded variation property?
\end{ques}

The second question deals with the second differentiable structure.
See \cite{Per94, Ot95, Ot97, LN19} for the Alexandrov/locally CAT case.

\begin{ques}
Let $X$ be a locally semiconcave G-space or a locally semiconvex, locally differentiable G-space as in Theorems \ref{thm:conc} and \ref{thm:conv}.
Does $X$ admit any second differentiable structure?
\end{ques}

The third question concerns the construction of a Finsler metric without local differentiability.

\begin{ques}\label{ques:conv}
Let $X$ be a locally semiconvex G-space as in Theorem \ref{thm:top} (not necessarily locally differentiable).
Does $X$ admit a continuous Finsler metric (almost everywhere in some sense)?
\end{ques}

On the other hand, as mentioned in Section \ref{sec:rel}, in view of the generalization of Alexandrov/locally CAT G-spaces, it is natural to add the Busemann concavity/convexity to our setting.
Recall Figures \ref{fig:alex} and \ref{fig:cat}.
From this point of view, we ask the following question.

\begin{ques}
Let $X$ be a locally semiconcave, locally Busemann concave G-space or a (locally differentiable) locally Busemann convex G-space.
Can Theorems \ref{thm:conc} and \ref{thm:conv} (as well as the answers to the above questions, if any) be improved in these cases?
\end{ques}

For example, Andreev \cite{An17nor} proved that every tangent cone of a locally Busemann convex G-space (without differentiability) is a finite-dimensional strictly convex normed space.
Some further results may be discussed in \cite{Fgeo}.
See also \cite{Ke19, HY25} for Busemann concave spaces.

Finally, we discuss the topological problems on G-spaces.
As mentioned in Section \ref{sec:rel}, it is a long-standing problem whether a G-space is a topological manifold.
Recall that, by the topological homogeneity of a G-space, it suffices to find at least one manifold point.

Berestovskii \cite{Be77} (cf.\ \cite{BHR11}) proved that any G-space such that every small metric ball is convex is finite-dimensional (see also Remark \ref{rem:dim2}).
The convexity of metric balls is considered as a qualitative generalization of the convexity of the distance function to a point.
Therefore, as a partial qualitative generalization of Theorem \ref{thm:top}, it is natural to ask the following.

\begin{prob}\label{prob:ball}
Let $X$ be a G-space such that every small metric ball is convex.
Show that $X$ is a topological manifold.
\end{prob}

As a special case, we can consider the following global problem.
Recall that a G-space $X$ is called \emph{straight} if one can take $U=X$ in Definition \ref{dfn:g}, i.e., every shortest path admits a unique extension to a shortest line.

\begin{prob}\label{prob:ball'}
Let $X$ be a straight G-space such that every metric ball is convex.
Show that $X$ is homeomorphic to Euclidean space.
\end{prob}

Here it is still sufficient to find one manifold point (\cite{Br61}; see \cite[Proof of Theorem 1.6]{FG25} for more details).
Several equivalent conditions for the convexity of metric balls in the case of straight G-spaces can be found in \cite[Theorem 20.9]{Bu55}.
In fact, Busemann \cite[Theorem 24.6]{Bu55} proved that under the additional assumptions of the parallel axiom and the differentiability of metric spheres, $X$ is isometric to a finite-dimensional normed space (for dimension $\ge 3$).
See also \cite{BP79} for another result in this direction.

Note that there is a G-space whose metric ball is not convex; see \cite{BP80}.
See also \cite{Fo04} for comparison of different notions of convexity in metric spaces.

\printbibliography

\end{document}